\documentclass[11pt,a4paper,reqno]{amsart}
\usepackage[applemac]{inputenc}
\usepackage[T1]{fontenc}
\usepackage{hyperref}
\usepackage{amsmath}
\usepackage{amsthm}
\usepackage{amsfonts}
\usepackage{amssymb}
\usepackage{graphicx}
\usepackage{color}
\usepackage{amsbsy}
\usepackage{mathrsfs}
\usepackage{bbm}
\usepackage{verbatim}
\usepackage{xcolor}
\usepackage{ulem}
\usepackage{dsfont}
\usepackage{graphicx}
\usepackage{subcaption}

\newtheorem{theorem}{Theorem}[section]
\newtheorem{lemma}[theorem]{Lemma}

\newtheorem{proposition}[theorem]{Proposition}
\newtheorem{corollary}[theorem]{Corollary}

\theoremstyle{remark}
\newtheorem{remark}[theorem]{\it \bf{Remark}\/}
\newenvironment{acknowledgement}{\noindent{\bf Acknowledgement.~}}{}
\numberwithin{equation}{section}
\catcode`@=11
\def\section{\@startsection{section}{1}%
	\z@{1.5\linespacing\@plus\linespacing}{.5\linespacing}%
	{\normalfont\bfseries\large\centering}}
\catcode`@=12
\newcommand{\be}{\begin{equation}}
	\newcommand{\ee}{\end{equation}}
\newcommand{\bea}{\begin{eqnarray}}
	\newcommand{\eea}{\end{eqnarray}}
\newcommand{\bee}{\begin{eqnarray*}}
	\newcommand{\eee}{\end{eqnarray*}}

\def\pa{\partial}

\def\na{\nabla}

\def\CC{\mathbb{C}}
\def\NN{\mathbb{N}}
\def\RR{\mathbb{R}}
\def\ZZ{\mathbb{Z}}

\def\eps{\vare}

\def\ep{\varepsilon}

\def\calB{{\mathcal B}}
\def\calA{{\mathcal A}}
\def\calD{{\mathcal D}}
\def\calE{{\mathcal E}}
\def\calF{{\mathcal F}}

\def\calH{{\mathcal H}}
\def\calI{{\mathcal I}}

\def\calL{{\mathcal L}}

\def\calT{{\mathcal T}}

\def\calY{{\mathcal Y}}

\def\calT{{\mathcal T}}

\catcode`@=11
\def\supess{\mathop{\operator@font Sup\,ess}}
\catcode`@=12
\def\CC{\mathbb{C}}

\def\NN{\mathbb{N}}

\def\RR{\mathbb{R}}
\def\CC{\mathbb{C}}

\def\ZZ{\mathbb{Z}}

\def\a{\alpha}
\def\e{\varepsilon}

\def\bar#1{{\overline #1}}

\def\R2+{\RR ^2_+}

\def\pa{\partial}
\def\na{\nabla}

\def\lim{\mathop{\rm lim}}

\def\sup{\mathop{\rm sup}}

\def\l{\lambda}

\def\log{{\rm log}}

\def\pa{\partial}

\def\pa{\partial}
\def\la{\langle}

\def\ra{\rangle}

\def\Dein{\Delta^{-1}}

\def\delg{\delta_g}

\def\eps{\varepsilon_s}
\def\epu{\varepsilon_u}
\def\dg{\delta_g}
\begin{document}
	
	\title[]{Nonradial Stability of self-similar blowup to Keller-Segel equation in three dimensions}

	\author[Z. Li]{Zexing Li}
	\address{Laboratoire AGM \\ CY Cergy Paris Universit\'e \\ 2 avenue Adolphe Chauvin \\ 95300 Pontoise \\ France}
	\email{zexing.li@u-cergy.fr}

 \author[T. Zhou]{Tao Zhou}
 \address{Department of Mathematics\\
National University of Singapore\\
Block S17\\
10 Lower Kent Ridge Road\\
Singapore\\
119076\\
Singapore}
\email{zhoutao@u.nus.edu}

	\maketitle

\begin{abstract}

In three dimensions, the parabolic-elliptic Keller-Segel system exhibits a rich variety of singularity formations. Notably, it admits an explicit self-similar blow-up solution whose radial stability, conjectured more than two decades ago in \cite{Brenner_Constantin_Leo_Schenkel_Venkataramani_steady_state_99}, was recently confirmed by Glogi\'c and Sch\"orkhuber \cite{MR4685953}. This paper aims to extend the radial stability to the nonradial setting, building on the finite-codimensional stability analysis in our previous work \cite{ksnsblowup}. The main input is the mode stability of the linearized operator, whose nonlocal nature presents essential challenges for the spectral analysis. Besides a quantitative perturbative analysis for the high spherical classes, we adapt in the first spherical class the wave operator method of Li-Wei-Zhang \cite{oseenli} for the fluid stability to localize the operator and remove the known unstable mode simultaneously. Our method provides localization beyond the partial mass variable and is independent of the explicit formula of the profile, so it potentially sheds light on other linear nonlocal problems. 
 \end{abstract}

 \section{Introduction}

We consider the coupled parabolic-elliptic Keller-Segel system
\be
\begin{cases}
\partial_t \rho  = \Delta \rho - \na \cdot (\rho \na c), \\
-\Delta c = \rho.
\end{cases}\tag{KS}
\label{equation, Keller-Segel}
\ee
This serves as a fundamental model for chemotaxis, where $\rho$ represents the cell density and $c$ denotes the concentration of the self-emitted chemical substance. For more background on chemotaxis and related models, interested readers can refer to \cite{MR3932458, MR2013508, MR2073515}.

The system has conservation of total mass $M(t) = \int \rho(t,x) dx$, and is invariant under spatial translation and the following scaling transformation:
if $(\rho,c)$ is the solution to \eqref{equation, Keller-Segel}, so is
\be
(\rho_\lambda,c_\lambda) = \left( \frac{1}{\lambda^2} \rho\left(\frac{t}{\lambda^2}, \frac{x}{\lambda} \right),  c\left( \frac{t}{\lambda^2}, \frac{x}{\lambda} \right) \right), \quad \l > 0.
\label{scaling invariant}
\ee

In this paper, we focus on the singularity formation mechanisms for the Keller-Segel system \eqref{equation, Keller-Segel} in three dimensions.

\mbox{}

\subsection{Background.}
\label{subsection111}

\subsubsection{Chemotaxis: global existence v.s. finite-time blowup.}

\mbox{}

 \vspace{0.1cm}

In two dimensions, the total mass is invariant under scaling symmetry, and the critical mass $M= 8\pi$ is a threshold between global existence and finite-time blowup. Specifically, if $M> 8 \pi$ and the initial data $\rho_0 \in L_+^1((1+|x|^2),dx)$, then the solution blows up in finite time. Conversely, if $M<8 \pi$, then the solution exists globally. For further details, see \cite{Blanchet_Dolbeault_Perthame_globalexistence06, Dolbeault_Perthame_globalexistence04}.

In three dimensions, in contrast, there is no critical mass: there exist radial blowup solutions with arbitrarily small mass \cite{MR1361006}. For general solutions that are not necessarily radial, \cite{MR2099126} provides sufficient conditions that determine blowup or global existence: blowup can occur when the initial data has a relatively small second moment, while global existence can be ensured by the smallness of $L^\frac{3}{2}$ norm of the initial data. Interested readers can refer to \cite{MR4201903, MR3411100, MR3438649, MR3936129} for more results.

\subsubsection{Singularity formation}

\mbox{}

\vspace{0.1cm}

When it comes to the singularity formation for the $2D$ Keller-Segel equation, with the model in the $L^1$ critical case, any finite-time blowup solution to \eqref{equation, Keller-Segel} is of type II\footnote{The solution of \eqref{equation, Keller-Segel} exhibits type I blowup at $t=T$ if 
\begin{equation*}
    \limsup_{t \to T} (T-t) \| \rho(t) \|_{L^\infty} < \infty,
\end{equation*}
otherwise, the blowup is of type II.
}, see \cite{Naito_Suzuki_typeIIblowup}. In particular, Rapha\"{e}l and Schweyer \cite{Raphael_Schweyer_2DtypeII_blowup14} precisely constructed a radially stable finite-blowup solution with the form
\be
\rho(t,x) \approx \frac{1}{\lambda^2(t)} U\left(\frac{x}{\lambda(t)} \right), \quad U(x) = \frac{8}{(1+|x|^2)^2}, \nonumber
\ee
where the blowup rate
\[
\lambda(t) = \sqrt{T-t} e^{-\sqrt{\frac{|\ln (T-t)|}{2}}+ O(1)}, \quad t \to T.
\]
Subsequently, Collot, Ghoul, Masmoudi, and Nguyen \cite{Collot_Ghoul_Masmoudi_Nguyen_2DtypeII_blowup22} extended to the non-radial stability and obtained a more precise blowup rate. In the same work \cite{Collot_Ghoul_Masmoudi_Nguyen_2DtypeII_blowup22}, they also constructed countably many unstable blow-up solutions. Later, Buseghin, D\'avila, Del Pino, and Musso \cite{buseghin2023existence}  
further extended to the multiple blowup bubbles without collision with a different strategy. Most recently, Collot, Ghoul, Masmoudi, and Nguyen \cite{2DKSmultibubble} discovered a new blowup mechanism where two bubbles would collapse and collide simultaneously.

Compared with $2D$ case, the singularity formation for the $3D$ Keller-Segel equation \eqref{equation, Keller-Segel} is much more diverse, with various blowup formations known to exist. When it comes to self-similar blowup case, a countable family of radial self-similar blowup profiles have been identified \cite{Brenner_Constantin_Leo_Schenkel_Venkataramani_steady_state_99, MR1651769}. In particular, there is an explicit self-similar solution given by
\be
\rho_*(t,x) = \frac{1}{T-t}Q\left( \frac{x}{\sqrt{T-t}} \right), \quad \text{ with } \quad Q(x) = \frac{4(6+|x|^2)}{(2 + |x|^2)^2},
\label{profile}
\tag{SS}
\ee
with the profile $Q$ satisfying the elliptic equation
\be
-\Delta Q + \frac 12 \Lambda Q - Q^2 - \na Q  \cdot \na \Dein Q = 0.
 \label{eqprofileQ}
\ee
This explicit self-similar solution was conjectured to be stable under the radial perturbation in \cite{Brenner_Constantin_Leo_Schenkel_Venkataramani_steady_state_99} based on numerical evidence, which was later rigorously established via semigroup theory by Glogi\'c and Sch\"orkhuber \cite{MR4685953}. Collot and Zhang \cite{collotzhangexicitedKS24} generalized their result to the low regularity and other self-similar profiles using the energy method and maximal principle; in the non-radial setting, the authors \cite{ksnsblowup} performed a finite-codimensional stability analysis of \eqref{profile} via an abstract semigroup method (see \cite{engel2000one,MR4359478}). Beyond self-similar blow-up solutions, other non-self-similar formations have also been identified. For example, Collot, Ghoul, Masmoudi, and Nguyen \cite{Collot_Ghoul_Masmoudi_Nguyen_3Dblowup_Collasping-ring_blowup23} discovered a type II blowup solution that concentrates in a thin layer outside the origin and collapses towards the origin. In addition, Nguyen, Nouaili, and Zaag \cite{nguyen2023construction} found a type I-Log blowup solution.

\subsection{Main result}

\mbox{}

The main result of this paper is the nonradial stability of \eqref{profile} building upon our previous finite-codimensional result \cite{ksnsblowup}. This generalizes the radial stability of \eqref{profile} established by Glogi\'c and Sch\"orkhuber \cite{MR4685953} to the non-radial setting.
 
\begin{theorem}[Stability of self-similar solution \eqref{profile} to $3D$ Keller-Segel equation]
\label{thm: stability of ss solu}
    There exists $\delta \ll 1$ such that for any given $\ep_0 \in H^2(\RR^3)$ with $\| \ep_0 \|_{H^2} \le \delta$, there exist $(\lambda_0,x_0) \in \RR^+ \times \RR^3$ with
    \[
    |\lambda_0 -1| + |x_0| \lesssim \| \e_0\|_{H^2}
    \]
    such that the initial data
    \[
    \rho_0 = Q + \ep_0,
    \]
    with $Q$ given in \eqref{profile}, generates a solution to \eqref{equation, Keller-Segel} blowing up at $T=\lambda_0^2$ with
    \begin{equation}
    \rho(t,x)= \frac{1}{\lambda_0^2 -t} \left( Q+ \ep \right) \left( t, \frac{x-x_0}{\sqrt{\lambda_0^2 -t}} \right),
   \label{fullstabblowup}
    \end{equation}
    where
    \begin{equation}
    \| \ep(t) \|_{H^2} \lesssim (T-t)^{\tilde \epsilon}, \quad \forall \; t \in (0,T),
    \label{smallnessremainingterm}
    \end{equation}
    for some $\tilde \epsilon >0$ independent of $\e_0$.
\end{theorem}

\mbox{}

\noindent \textit{Comments on Theorem \ref{thm: stability of ss solu}}

\mbox{}

\noindent \textit{1. From radial to nonradial stability.}
Most of the $3D$ results mentioned previously for construction and stability of blowup are in the radial setting, which is largely due to the \textit{partial mass variable} 
\be
m_\rho(t, r) := \int_{|x| \le r} \rho(t, x) dx = 4\pi \int_0^{r} \rho(t, s)s^2 ds, 
\label{partial mass}
\ee
 transforming the original system \eqref{equation, Keller-Segel} into a \textit{local} scalar equation. In the absence of radial symmetry, the \textit{nonlocal nature} could create great challenges for spectral analysis, matching asymptotics, maximum principle and so on. Nevertheless, understanding the nonradial setting should be natural and essential, particularly in that many chemotaxis models involve anisotropic physical effects which break the radial symmetry, such as the buoyancy force from fluids \cite{hukiselevyao2023suppression,ksnsblowup}.
 
As a first step towards nonradial stability, our finite-codimensional analysis \cite{ksnsblowup} exploited that the functional framework of \cite{engel2000one,MR4359478} could survive nonlocal perturbation. In this work, we resolve the spectral analysis for the nonlocal and non-self-adjoint linearized operator (Theorem \ref{thm: mode stability main}) and conclude the nonradial stability. 

\mbox{}

\noindent \textit{2. Extension of Theorem \ref{thm: stability of ss solu}.} Our method is robust enough so that the result can be generalized in the following various senses.

\mbox{}

\vspace{-0.3cm}

\noindent \textit{$(1)$ Regularity.} The choice of $H^2$ regularity guarantees simple nonlinear estimates thanks to Sobolev embedding $H^2(\RR^3) \hookrightarrow L^\infty(\RR^3)$, and can be replaced by any $H^k$ space with $k \ge 2$ with slight modifications to the proof. However, it is not obvious to obtain $L^\infty$ stability as in Collot-Zhang \cite{collotzhangexicitedKS24} via our current method. 

\mbox{}

\vspace{-0.3cm}

\noindent \textit{$(2)$ Other profiles.} For other radial self-similar profiles from \cite{Brenner_Constantin_Leo_Schenkel_Venkataramani_steady_state_99, MR1651769} solving \eqref{eqprofileQ}, our analysis should be applicable with an numerical approximate profile to perform quantitative estimates. In particular, the wave operator method requires a non-vanishing condition, which is supported by numerical evidence (see Remark \ref{rmkprof}). Consequently, we expect nonradial finite-codimensional nonlinear stability boils down to radial finite-codimensional mode stability as required in Collot-Zhang \cite{collotzhangexicitedKS24}.

Besides, for the Keller-Segel equation \eqref{equation, Keller-Segel} in higher dimensions $N \ge 3$, similar to \eqref{profile}, there exists an explicit self-similar blowup solution
\[
\rho_N(t,x) = \frac{1}{T-t} Q_N \left( \frac{x}{\sqrt{T-t}} \right),
\text{ where } Q_N(x) = \frac{4(N-2)(2N+|x|^2)}{(2(N-2) + |x|^2)^2}, \quad x \in \RR^N.
\]
One should be able to establish the nonradial stability of $\rho_N$ following the same argument. 

\mbox{}

\vspace{-0.3cm}

 \noindent \textit{$(3)$ Application to other PDE models.} 
The nonradial stability analysis of \eqref{profile} establishing Theorem \ref{thm: stability of ss solu} can be applied to other PDE models, including the $3D$ Keller-Segel-Navier-Stokes system with buoyancy force \cite{ksnsblowup} and possibly other subcritical perturbation of \eqref{equation, Keller-Segel}.

\mbox{}
\subsection{Mode stability}
\label{subsec: mode stability}

One major input of this paper is an exact counting of unstable directions for the nonlocal linearized operator \eqref{eqdefcalL}. The main result is stated as follows.

\begin{theorem}[Mode stability]
\label{thm: mode stability main}
    For $k \ge 0$ and the linearized operator $\calL$ given in \eqref{eqdefcalL}, there exists $0 <\tilde{\delta} \ll1$ such that
    \be
    \sigma_{disc}\left(\calL \big|_{H^k(\RR^3)} \right) \cap \{ z \in \CC: \Re z < \tilde{\delta} \} = \Big\{ -1 , -\frac{1}{2} \Big\},
    \label{discrete unstable eigen of L}
    \ee
    with the related generalized unstable eigenspaces as
    \be
    % \text{span} \{ \Lambda Q, \partial_{x_1} Q, \partial_{x_2} Q, \partial_{x_3} Q \}.
    \bigcup_{m \ge 1} \ker (\calL-\l)^{m} = \ker (\calL-\l) = \begin{cases}
        \text{span}\{ \partial_{x_1} Q,
         \partial_{x_2} Q, \partial_{x_3} Q \},  & \text{ when } \l = -\frac 12, \\
         \text{span}\{ \Lambda Q \}, & \text{ when } \l = -1.
    \end{cases} \label{equnstabspec2}
    \ee
\end{theorem}
\mbox{}

\noindent \textit{Comments on Theorem \ref{thm: mode stability main}.}

\mbox{}

\noindent \textit{$1.$ Mode stability of self-similar blowup solution.}
Briefly, mode stability means that all unstable modes arise from symmetries of the original equation, so that they are removable by choosing a correct reference frame during the nonlinear evolution. Thus it is a crucial step towards nonlinear stability. In our case, we prove every unstable mode of $\calL$ is generated either by translation invariance or scaling invariance, which give rise to $\nabla Q$ or $\Lambda Q$ respectively (see \eqref{eqexplicitunstablemode}). 

Various methods have been developed for the mode stability, or more generally, characterization of the spectrum of certain linearized operators: for local parabolic models, the linearized operator could be turned into a Schr\"odinger operator $-\Delta + V$, so that one can exploit Sturm-Liouville theory \cite{MR4400117,MR4073868,MR3986939} or supersymmetry method and GGMT bound \cite{MR4126325,MR4685953}\footnote{Among them includes the nonlocal model \eqref{equation, Keller-Segel} in 2D \cite{MR4400117} for non-radial data and in 3D \cite{MR4685953} for radial data. However, in the radial case, the partial mass variable \eqref{partial mass} localizes $\calL$; while in the 2D non-radial case, with the blowup profile being the ground state, the authors exploited its variational information and self-adjoint structure.}; for wave-type models, \cite{MR3623242,MR3475668,zbMATH07928637,MR3537340} used quasi-solution method or explicit hypergeometric functions; and for Schr\"odinger type equations, one has Jost function argument and linear Liouville argument in the bifurcation regime \cite{Limodestability,MR1852922}. However, these approaches are more or less relying on ODE analysis and hence merely effective for local differential operators. Extending them to nonlocal operators remains a significant challenge. 

Besides, regardless of nonlocality, variational properties can lead to spectral information, for example for the linearized operators of ground states of nonlinear Schr\"{o}dinger equations, Hartree equations \cite{lemzmannHartee09, Nakanishi_Schlag_invariant_manifold, WeinsteinlinearizedopofNLS85}, and 2D Keller-Segel equation under the blowup renormalization \cite{MR4400117}. Unfortunately, it is widely open to find variational characterization for self-similar profiles, in particular, unclear for \eqref{profile}. 

Another possible approach to mode stability of a specific nonlocal operator is matrix approximation through truncation and discretization, and then computing its spectrum numerically and estimating the error rigorously (see \cite{MR2214946}). Clearly, this would rely on heavy numerical assistance. 

\mbox{}

\noindent \textit{$2.$ Idea of the proof: perturbative analysis and wave operator method.}

In this paper, we propose an analytical method that fully exploits the structure of the equation \eqref{equation, Keller-Segel}, with the mere computer assistance used for evaluating two integrals \eqref{eqNpleff}. Thanks to the radial symmetry of $Q$, we restrict $\calL$ into spherical classes and analyze each class $l \ge 1$. The radial case $l = 0$ was done in \cite{MR4685953}.

Since $-\Delta$ is more coercive in high spherical classes, we treat the nonlocal part perturbatively for $l \ge 2$ and use arguments for Schr\"odinger operators. The $l = 2$ case requires more delicate treatment, involving a partial localization and GGMT bound.

The most challenging case is $l=1$, due to the presence of known eigenvalue and weak coercivity of $-\Delta$. Inspired by the \textit{wave operator method} introduced by Li-Wei-Zhang \cite{oseenli}, which was originally used to study the linear stability of rotating Oseen vortices under the $2D$ Navier-Stokes flow, we construct a \textit{formal wave operator} that simultaneously localizes the operator and removes the known unstable mode (see Proposition \ref{propwaveop2}), and this allows us to conclude the mode stability when $l = 1$. Readers can refer to Remark \ref{rmkwom} for comparison of these two settings. Besides, the wave operator could be constructed implicitly in more complicated scenario \cite{zbMATH07213543,zbMATH07154876}. 

We mention two special features of our application:

\mbox{}
\vspace{-0.35cm}

$(1)$ \textit{Substitution of localization and supersymmetry method:} 
Such construction works for not only $l = 1$ but also the radial case $l=0$ (Remark \ref{rmkradial}), while the partial mass localization only functions for $l = 0$. Hence we expect this structure might be useful for the linear nonlocal problems when direct localization is impossible. Interestingly, the existence of the unstable mode seems helpful for the localization. 

\mbox{}
\vspace{-0.35cm}

$(2)$ \textit{Independence of the explicit formula of $Q$.}
Although the construction in \cite{oseenli} seems magical and strongly relies on the special form of rotating Oseen vortices,  the construction of wave operator in our case depends only on the profile equation \eqref{eqprofileQ} and a non-vanishing condition (Remark \ref{rmkprof}), rather than on the explicit expression of $Q$. Therefore, the method is potentially flexible for other self-similar profiles or other nonlocal PDE models.

\subsection{Sketch of proof.} \label{sec14}

In this section, we will give the sketch of proof of Theorem \ref{thm: stability of ss solu} and Theorem \ref{thm: mode stability main}.

\vspace{-0.2cm}

\mbox{}

\noindent \textit{1. Set up: Self-similar renormalization and linearization.} 

\mbox{}

 \vspace{-0.3cm}

To study the stability of the self-similar blowup solution \eqref{profile}, we renormalize the system \eqref{equation, Keller-Segel} in the self-similar coordinate and consider the long time stability of $Q$ under the renormalized equation \cite{zbMATH03937666,zbMATH03969175,giga1989nondegeneracy}. Precisely, let
 \be 
 \rho(t,x)=\frac{1}{\mu^2} \Psi(\tau,y), \quad  y = \frac{x}{\mu}, \quad  \frac{d\tau}{dt} = \frac{1}{\mu^2}, \quad \tau\Big|_{t=0} =0, \quad \frac{\mu_\tau}{\mu} = -\frac{1}{2},
 \label{sscoordi}
 \ee
 then the original system \eqref{equation, Keller-Segel} is mapped to the following renormalized equation
 \be
 \partial_\tau \Psi + \frac{1}{2} \Lambda \Psi= \Delta \Psi + \na \cdot (\Psi \na \Dein \Psi),
 \label{eqrenormalized}
 \ee
 where 
 \be 
\Lambda = \Lambda_0 + \frac{1}{2}, \quad { with } \quad \Lambda_0 = y \cdot \na + \frac{3}{2}.
\label{scalingopdef}
\ee
It is easy to check that, $Q$, the self-similar profile given by \eqref{profile}, is the steady state of the renormalized equation \eqref{eqrenormalized}. 

Plugging the ansatz $\Psi = Q + \ep$, the error term $\ep$ solves the following equation
\be
\partial_\tau \ep = -\calL \ep + N(\ep),
\label{linearizedeq}
\ee
with $\calL$ the linearized operator defined by
\be
\calL \e = - \Delta \e + \frac{1}{2} \Lambda \e - 2 Q \e - \nabla \Dein Q \cdot \nabla \e - \na Q \cdot \nabla \Dein \e, \label{eqdefcalL} 
\ee
and $N(\ep)$ is the nonlinear term defined by
\be
N(\ep) = \na \cdot (\ep \na \Delta^{-1} \ep).
\label{nonlinearterm}
\ee
From \cite[Proposition 2.3, Remark 2.4]{ksnsblowup}, $\calL$ is a closed operator on $H^k(\RR^3)$ for all $k \ge 0$, and it possesses \textit{explicit unstable eigenmodes}
\be 
\calL (\Lambda Q) =- \Lambda Q,\quad \calL(\partial_{x_j} Q) = -\frac{1}{2} \partial_{x_j} Q,\quad {\rm for }\,\, j = 1, 2, 3,  \label{eqexplicitunstablemode}
\ee
obtained by
applying scaling transformation \eqref{scaling invariant} and spatial translation to the profile equation \eqref{eqprofileQ}. 

\mbox{}

\vspace{-0.2cm}

\noindent \textit{2. Proof of mode stability Theorem \ref{thm: mode stability main}.} 

\mbox{}

 \vspace{-0.3cm}

 Section \ref{sectionmodesta} is dedicated to the proof of Theorem \ref{thm: mode stability main}. In Section \ref{sec21}, we begin by restricting $\calL$ onto each spherical class (see Lemma \ref{thmLlresonHl}) and investigate the unstable spectrum of each $\calL_l$ separately.

 For $\calL_l$ with $l \ge 2$, we will verify that the nonlocal effect of $\calL_l$ should be dominated by the sufficiently strong positivity coming from the angular momentum term $\frac{l(l+1)}{r^2}$ for all $l \ge 2$. More specifically, in Section \ref{subsecl3}, we will establish the coercivity of $\calL_l$ with $l \ge 3$ by direct inner product estimate (see Proposition \ref{propmodstabl3}).
 As for the borderline scenario $l=2$, in Section \ref{subsecl2}, we employ a generalization of the partial mass transformation to extract more coercivity from $\Delta_l$ and refine the inner product estimate to GGMT bound (refer to Lemma \ref{lemsymestnon} and Theorem \ref{thmGGMT}). 

 For $\calL_1$, in Section \ref{sec24}, we construct in Proposition \ref{propwaveop2} a \textit{formal wave operator} that simultaneously localizes the operator and removes the known unstable mode. The conjugated operator $\tilde \calL_1$ (see \eqref{localLl1}) is then clearly coercive.

 Finally, we end the proof of Theorem \ref{thm: mode stability main} in Section \ref{subseccompleteproofL} by applying the coercivity estimates and a priori improvement of regularity for unstable eigenfunctions. Some necessary estimates are recorded in Appendix \ref{appA} and Appendix \ref{appB}. In Section \ref{sec26}, we record the linear theory of $\calL$ in Proposition \ref{proposition: spectral properties of L} as a corollary of Theorem \ref{thm: mode stability main} combined with the abstract semigroup theory established in \cite{ksnsblowup}.

\mbox{}

\vspace{-0.2cm}

\noindent \textit{3. Proof of nonlinear stability Theorem \ref{thm: stability of ss solu}.}

\mbox{}

\vspace{-0.3cm}

In Section \ref{sectionnonsta}, we first prove the nonlinear stability of $Q$ with sharp codimension $4$, and construct a Lipschitz continuous stable manifold. Using the four dimensional symmetry group, one can shift the initial data onto the stable manifold with the help of Lipschitz continuity, concluding the stability of \eqref{profile} without losing codimensions. This matching initial data argument comes from \cite{li2023stability}.

\subsection{Notation}

Throughout the paper, we use the notation $A\lesssim B$ to denote that there exists a constant $C>0$ such that $0 \le A \le C B$. Similarly, $A \sim B$ means that there exist constant $0<c<C$ such that $cA \le B \le CA$. We also denote by
\[
\la r \ra = \sqrt{1+r^2}.
\]

Next, we denote $\RR^+ := (0, \infty)$. And given an interval $I \subset \RR^+$ and a weight function $\omega: I \to \RR$, we define the $L_w^2(I)$ weighted inner product by
    \[ (f, g)_{L^2_\omega(I)} = \int_I f \bar g \omega dr, \]
and the weighted $L^2$ space $L^2_\omega(I)$ by the collection of all functions $f$ satisfying $\| f \|_{L_w^2(I)}^2 : = (f,f)_{L_w^2(I)} < \infty$. Moreover, we define $C_c^\infty(I)$ as the collection of all  
 $C^\infty(I)$ functions with compact support in $I$. 
    
We call the vector of form $\alpha = (\alpha_1, \alpha_2, \alpha_3) \in \NN^3$ the multi-index of order $|\alpha|=\alpha_1 + \alpha_2 +\alpha_3$. For any given scalar function $f,g$, we define the partial derivative of $f$ with respect to the multi-index $\alpha$ by
\[
\partial^\alpha f(x) = \partial_{x_1}^{\alpha_1} \partial_{x_2}^{\alpha_2} \partial_{x_3}^{\alpha_3} f(x),
\]
and for any $k \ge 0$, we denote $\calD^k$ by
\[
\calD^k f = (\partial^\alpha f)_{|\alpha|=k}.
\]
In particular, for $k=1$, we simplify the notation into $\calD f =(\partial_{x_1} f, \partial_{x_2}f , \partial_{x_3} f)$, which is the gradient of $f$. Hence for any integer $k \ge 0$, we define the inner product on $\dot H^k$ and $H^k$ by 
\[ (f,g)_{\dot H^k} = (\calD^k f, \calD^k g)_{L^2} := \sum_{|\alpha|=k}(\partial^\alpha f, \partial^\alpha g)_{L^2},\quad (f,g)_{H^k} = (f,g)_{L^2} +  (f,g)_{\dot H^k}.\]
And $H^k$ (or $\dot H^k$) is the collection of all functions with finite $H^k$ (or $\dot H^k$) norm. We also write $H^\infty = \cap_{k \ge 0} H^k$. In addition, to avoid confusion, we denote $H^{k} (\RR^3; \RR)$, $H^k(\RR^3;\CC)$ (or $\dot H^{k} (\RR^3; \RR)$, $\dot H^k (\RR^3; \CC)$) for the collection of all $\RR$-valued or $\CC$-valued functions respectively with finite $H^k$ (or $\dot H^k$) norm. Out of the simplicity of notation, we will use $H^k(\RR^3)$ (or $\dot H^k(\RR^3)$) to refer $H^k (\RR^3; \CC)$ (or $\dot H^k (\RR^3; \CC)$) in Section \ref{sectionmodesta} and refer $H^k (\RR^3; \RR)$ (or $\dot H^k (\RR^3; \RR)$) in Section \ref{sectionnonsta}.

In addition, for any $k \in \mathbb{Z}$, we define the differential operator $D_k$ by
\be 
D_k := \pa_r + \frac{k}{r} = r^{-k} \pa_r r^k. 
\label{Dkdef}
\ee
We also define
\be 
D_k^{-1} = \left\{ \begin{array}{ll}
   r^{-k} \int_0^r f s^k ds,  & k > 0, \\
   - r^{-k} \int_r^\infty f s^k ds,  & k \le 0.
\end{array} \right. \label{eqdefDk-1}  
\ee

Furthermore, for a Banach space $X$, for given $x_0 \in X$ and $\epsilon > 0$, we define an open ball in $X$ with center $x_0$ and radius $\epsilon$ by
    \[ B^X_\epsilon(x_0) = \{ x \in  X: \| x - x_0 \|_{X} < \epsilon \}.   \]
%And we write $B^X_\epsilon := B^X_\epsilon(0)$. 

Finally, if $A$ is a linear operator on a Hilbert space $H$, then we denote $\rho(A)$, $\sigma(A), \sigma_{disc}(A)$ to be the resolvent set, spectral set of $A$, discrete spectral set of $A$ respectively. 

\begin{acknowledgement}
The work of Z.L. is part of the ERC starting grant project FloWAS that has received funding from the European Research Council (ERC) under the Horizon Europe research and innovation program (Grant agreement No. 101117820). Z.L. is also partially supported by the ERC advanced grant SWAT. T.Z. is partially supported by MOE Tier $1$ grant  A-0008491-00-00. The authors gratefully thank Te Li and Yao Yao for the valuable discussions and suggestions about this work. 
\end{acknowledgement}

 \section{Mode stability and linear theory}
 \label{sectionmodesta}

 \subsection{Preliminaries} 
 \label{sec21}
 First of all, we recall $L^2(\RR^3)$ direct decomposition using spherical harmonics. From \cite[Theorem $3.5.7$]{harmonicsimon},
\be
L^2(\RR^3) = \bigoplus_{l=0}^\infty \calH_{(l)}, \quad {\rm where}\quad \calH_{(l)} :=  L^2(\RR^+, r^2 dr) \otimes \calY_{(l)},
\label{directdecomofL2}
\ee
where $\calY_{(l)} = \text{span}\{ Y_{lm} \}_{m=-l}^l$ denotes the $(2l+1)$ dimensional eigenspace corresponding to the eigenvalue $-l(l+1)$ of the spherical Laplacian $\Delta_{\mathbb{S}^2}$ acting on $L^2(\mathbb{S}^2)$ with $\{ Y_{lm} \}_{l \ge 0, |m| \le l}$ as the orthonormal basis of $L^2(\mathbb{S}^2)$. Therefore, for any $u \in L^2(\RR^3)$, we can decompose it into
\be
u(x) = \sum_{l=0}^\infty u_{l}(x)  = \sum_{l=0}^\infty \sum_{m=-l}^{l} u_{lm}(r) Y_{lm}(\Omega),
\label{sphercial harmonic decomposition: f}
\ee
where $x = r\Omega$ with $r= |x|$ and $\Omega \in \mathbb{S}^2$, $u_{l} \in \calH_{(l)}$ and $u_{lm} \in L^2(\RR^+,r^2 dr)$, and the following Pythagorean formula satisfies:
\[
\| u \|_{L^2(\RR^3)}^2 = \sum_{l=0}^\infty \| u_l \|_{L^2(\RR^3)}^2
= \sum_{l=0}^\infty \sum_{m=-l}^l \| u_{lm} \|_{L^2(\RR^+; r^2 dr)}^2.
\]
In addition, if $u \in H^k(\RR^3)$ with $k \ge 0$, for each component $u_{lm}(r) Y_{lm}(\Omega)$ in \eqref{sphercial harmonic decomposition: f} must lie in $H^k$ as well, and satisfy 
\[
\| u \|_{\dot H^k(\RR^3)}^2 = \sum_{l=0}^\infty \| u_l \|_{\dot H^k(\RR^3)}^2
= \sum_{l=0}^\infty \sum_{m=-l}^l \big\| u_{lm}(r) Y_{lm}(\Omega) \big\|_{\dot H^k(\RR^3)}^2,
\]
since Fourier transform preserves spherical harmonic class (see \cite[Theorem 3.10]{MR0304972}). 
And in later discussion, we always regard $u_{lm}(r)$ as a radial function on $\RR^3$, and thus $u_{lm} \in L^2(\RR^3)$.

 With the direct decomposition of $L^2(\RR^3)$ space given in \eqref{directdecomofL2}, to prove Theorem \ref{thm: mode stability main}, the main idea is to consider the behavior of operator $\calL$ restricted on each $\calH_{(l)}$. Now we will prove that the linearized operator $\calL$ acts invariantly on $\calH_{(l)}$ and find an explicit form of restriction of $\calL$ on each $\calH_{(l)}$. 
 \begin{lemma} \label{thmLlresonHl}
 The linearized operator given in \eqref{eqdefcalL} acts invariantly on $\calH_{(l)}$. For all $k \ge 0$, and $u \in  D \left( \calL \Big|_{H^k} \right)$, with the decomposition \eqref{sphercial harmonic decomposition: f}, we have
 \[
 \calL u = \sum_{l=0}^\infty \sum_{m=-l}^l  (\calL_l u_{lm}(r)) Y_{lm}(\Omega),
 \]
where $\calL_{l}$ can be explicitly written by
 \be
 \calL_l f =- \Delta_{l} f + \frac{1}{2} \Lambda f -2 Qf - \partial_r \Delta^{-1} Q \cdot  \partial_r f - \partial_r Q \cdot \partial_r \Delta_l^{-1} f,
\label{eqdefcalLl}
 \ee
 with
\be
 -\Delta_{l} = -\partial_r^2 - \frac{2}{r} \partial_r + \frac{l(l+1)}{r^2},
 \label{Laplace on Hl}
 \ee
 and
 \be
\Delta_{l}^{-1} f= -\frac{1}{2l+1} \int_0^\infty \frac{r_<^l}{r_>^{l+1}} f(r') r'^2 dr.
\label{eqdefDeltal-1}
\ee
 \end{lemma}

 \begin{proof}
 Recall \cite[Theorem $3.5.8$]{harmonicsimon}, we know that $-\Delta$ acts invariantly on each $\calH_{(l)}$ and the restriction $\Delta_{l}$ satisfies \eqref{Laplace on Hl}. In addition, since for any spherically harmonic function $Y_{lm}$ is a function on the unit sphere $\mathbb{S}^2$, $\na Y_{lm}$ is a vector tangent to the sphere, which thus implies that 
\[
x \cdot \na Y_{lm} =0, \quad \text{ for any } l\ge 0 \text{ and }  -l \le m \le l.
\]
Hence $x\cdot \na $ acts invariantly on $\calH_{(l)}$ and the related restriction should be $r \partial_r$.

 Then it suffices to consider the behavior of $\Delta^{-1}$. From the \cite[Example $3.5.12$]{harmonicsimon}, we have the multipole expansion
 \[
 \frac{1}{|x-x'|} = 4 \pi \sum_{l=0}^\infty \sum_{m=-l}^l \frac{1}{2l+1} \frac{r_<^l}{r_>^{l+1}} Y_{lm}(\Omega) \bar{Y_{lm}}(\Omega'),
 \]
 where $r_<= \min\{ |x|, |x'| \}$ and $r_> = \max\{ |x|, |x'| \}$. Thus
 \begin{align*}
     & \quad \Delta^{-1} u 
     = -\frac{1}{4 \pi} \int \frac{1}{|x-x'|} u(x') dx'
      =  - \int_0^\infty \left( \int_{\partial B_1} \frac{1}{4 \pi |x-x'|} u(x') d\Omega' \right) r'^2 dr' \\
     & = -\int_0^\infty  \int_{\pa B_1} \left( \sum_{l=0}^\infty \sum_{m=-l}^l \frac{1}{2l+1} \frac{r_<^l}{r_>^{l+1}} Y_{lm}(\Omega) \overline {Y_{lm}(\Omega')} \right) \left(\sum_{l'=0}^\infty \sum_{m'=-l'}^{l'} u_{l'm'}(r') Y_{l'm'}(\Omega') \right) d\Omega' r'^2 dr' \\
     &=- \sum_{l=0}^\infty \sum_{m=-l}^l \frac{1}{2l+1}\left( \int_0^\infty \frac{r_<^l}{r_>^{l+1}} u_{lm}(r')r'^2 dr' \right) Y_{lm}(\Omega), 
 \end{align*}
which implies that $\Delta^{-1}$ acts invariantly on $\calH_{(l)}$ and the related restricted operator $\Delta_{l}^{-1}$ should be \eqref{eqdefDeltal-1}.
\end{proof}

Furthermore, we record some elementary decomposition formulas for $\Delta_l$ and $\Delta_l^{-1}$ into first order operators.
 \begin{lemma} For $l \ge 0$,
    \bea
    \Delta_l &=& D_{l+2} D_{-l},\qquad
    \Delta_l^{-1} \,\,=\,\, D_{-l}^{-1} D_{l+2}^{-1},  \label{eqDeltalDl} \\
 \pa_r \Delta_l^{-1} &=& \frac{1}{2l+1} \left[ (l+1) D_{l+2}^{-1} + l D_{-(l-1)}^{-1} \right], \label{eqDeltalintop1} \\
 \Delta_l^{-1}\pa_r &=& \frac{1}{2l+1} \left[ (l+2) D_{l+1}^{-1} + (l-1) D_{-l}^{-1} \right], \label{eqDeltalintop2}
\eea
where $D_{k}$ and $D_{k}^{-1}$ are defined in \eqref{Dkdef} and \eqref{eqdefDk-1}.
\end{lemma}
\begin{proof}
    It is straightforward to verify these relations using \eqref{Laplace on Hl}-\eqref{eqdefDeltal-1}.
\end{proof}

 \subsection{Mode stability for $l \ge 3$} \label{subsecl3}
 In this section, we are devoted to the easiest case of the mode stability of $\calL_l$ with $l \ge 3$. The main result is as follows:
\begin{proposition}  \label{propmodstabl3} For $l \ge 3$, for any $f \in H^\infty(\RR^3) \cap  \calH_{(l)}$, we have the following coercivity: 
    \be \Re(\calL f, f)_{L^2(\RR^3)} \ge \frac 18 \| f \|_{L^2(\RR^3)}^2.
    \label{Llposivwith3}
    \ee
    % In particular, 
    % \be \sigma\left(\calL_l \big|_{L^2} \right) \cap  \{ \l \in \CC: \Re \l < \frac 18 \} = \emptyset  \ee
\end{proposition}

We split into the following several steps to prove the Proposition \ref{propmodstabl3}. First of all, if $f$ is of form $f(x) = f_{lm}(r) Y_{lm}(\Omega) \in H^\infty \cap \calH_{(l)}$, then
 we can obtain the following inner product form:
 \begin{lemma} \label{lemhighsphcomp} For any $f = f_{lm}(r) Y_{lm}(\Omega) \in H^\infty(\RR^3)\cap \calH_{(l)}$ with $l \ge 0$, $|m| \le l$, we have
    \be 
    \begin{split}
  & \quad 4 \pi \Re (\calL f, f)_{L^2(\RR^3)}  \\
  &= \| \pa_r f_{lm} \|_{L^2(\RR^3)}^2 +l(l+1) \| r^{-1} f_{lm} \|_{L^2(\RR^3)}^2 + \frac 14 \| f_{lm} \|_{L^2(\RR^3)}^2  
  - \frac 32 (Qf_{lm}, f_{lm})_{L^2(\RR^3)} \\
  & \quad + \frac 12 \int_{\RR^3} |\pa_r \Delta_l^{-1} f_{lm} |^2 \left[(\pa_r - 2r^{-1}) \pa_r Q  \right] dx - \frac{l(l+1)}{2} \int_{\RR^3} |\Delta_l^{-1} f_{lm}|^2 r^{-2} \pa_r^2 Q dx.
     \end{split} \label{eqquacalLl}
    \ee
\end{lemma}
\begin{proof}
    Recall $\Lambda = \Lambda_0 + \frac 12$ with $\Lambda_0$ skew-symmetric on $L^2$. The first four terms of $\calL f$ in \eqref{eqdefcalL} has the following contributions:
\begin{align*}
    &  \Re (-\Delta f, f)_{L^2  } =  \| \nabla f \|_{L^2 }^2, \qquad  \frac 12 \Re (\Lambda_0 f, f)_{L^2 } =0, \\
    & \Re \left(\left(\frac 14 - 2 Q\right)f, f\right)_{L^2 } = \frac 14 \| f \|_{L^2 }^2 - 2 (Q f , f)_{L^2 }, \\
    &  \Re (-\nabla \Delta^{-1} Q \cdot \nabla f, f)_{L^2 } =-\frac 12 \int \nabla |f|^2 \cdot    \nabla \Delta^{-1} Q dx= 
    \frac 12 (Qf, f)_{L^2 }.
\end{align*}
Furthermore, 
\begin{align}
\| \na (f_{lm} Y_{lm}) \|_{L^2(\RR^3)}^2
& = \left(-\Delta (f_{lm} Y_{lm}), f_{lm} Y_{lm} 
\right)_{L^2}    
= (-\Delta_l f_{lm} Y_{lm}, f_{lm} Y_{lm})_{L^2} \notag \\
& = \int_0^\infty \left(-\partial_r^2 - \frac{2}{r} \partial_r f_{lm} + \frac{l(l+1)}{r^2} f_{lm} \right) f_{lm} r^2 dr \notag\\
& = \int_0^\infty |\partial_r f_{lm}|^2 r^2 dr + l(l+1) \int |f_{lm}|^2 dr \notag\\
& = \frac{1}{4 \pi} \| \partial_r f_{lm} \|_{L^2(\RR^3)}^2 + \frac{l(l+1)}{4 \pi} \| \frac{f_{lm}}{r} \|_{L^2(\RR^3)}^2,
\label{nablaflm}
\end{align}
and
\begin{align*}
    \| f\|_{L^2}^2  = \frac{1}{4 \pi} \| f_{lm} \|_{L^2(\RR^3)}^2 \text{ and }
    (Qf,f)_{L^2} = \frac{1}{4 \pi}(Qf_{lm}, f_{lm})_{L^2(\RR^3)}.
\end{align*}

For the last term in \eqref{eqdefcalL}, we denote $w = \Delta^{-1} f$ so $w = w_{lm} Y_{lm}$ with $w_{lm} = \Delta_l^{-1} f_{lm}$, then
\begin{align*}
    & \quad \Re (-\nabla Q \cdot \nabla \Delta^{-1} f, f)_{L^2  }
     = -\Re (\na Q \cdot \na w, \Delta w)_{L^2}
    = -\Re \left(\partial_r Q \partial_r w_{lm} Y_{lm},    \Delta_l w_{lm} Y_{lm} \right)_{L^2} \\
    & = -\int_0^\infty \partial_r Q \cdot \partial_r w_{lm} \cdot  \Delta_l \bar{w}_{lm} \cdot r^2    \left( \int_{\mathbb{S}^2} Y_{lm}(\Omega) \bar{Y}_{lm}(\Omega)  d\Omega \right) dr \\
    & = -\int_0^\infty \partial_r Q \cdot \partial_r w_{lm} \cdot  \left( \partial_r^2 \bar w_{lm} + \frac{2}{r} \partial_r \bar w_{lm} - \frac{l(l+1)}{r^2} \bar w_{lm} \right) r^2     dr \\
    & = -\int_0^\infty \left[ \frac{1}{2} \partial_r| 
    \partial_r w_{lm}|^2+ \frac{2}{r}|\partial_r w_{lm}|^2 -  \frac{l(l+1)}{2r^2} \partial_r|w_{lm}|^2 \right] \partial_r Q \cdot    r^2 dr \\
    & =  \frac{1}{2}\int_0^\infty |\partial_r w_{lm}|^2 [ ( \partial_r - 2r^{-1} ) (   \partial_r Q)] r^2 dr -  \frac{l(l+1)}{2} \int_0^\infty  |w_{lm}|^2 r^{-2} \partial_r(   \partial_r Q) r^2 dr \\
    & = \frac{1}{4 \pi} \cdot \frac{1}{2}\int_{\RR^3}|\partial_r w_{lm}|^2 [ ( \partial_r - 2r^{-1} ) (   \partial_r Q)] dx - 
    \frac{1}{4 \pi} \cdot \frac{l(l+1)}{2} 
    \int_{\RR^3}  |w_{lm}|^2 r^{-2} \partial_r(   \partial_r Q)  dx
     % \\ & =\frac{1}{2}\int_{\RR^3}|\partial_r \Delta^{-1} f|^2 [ ( \partial_r - 2r^{-1} ) (   \partial_r Q)] dx - 
     %   \frac{l(l+1)}{2} 
     % \int_{\RR^3}  |\Delta^{-1} f|^2 r^{-2} \partial_r^2 Q dx.
\end{align*}
Summing up these contributions concludes \eqref{eqquacalLl}.
\end{proof}

Next, we provide some pointwise estimates of some related functions, which will help us to explore more about the right hand side of \eqref{eqquacalLl}.
\begin{lemma} The profile $Q$ given by \eqref{profile} satisfies 
   \be
     Q \le \frac 92 r^{-2}, \quad \partial_r^2 Q \le \frac{136}{3}(2+r^2)^{-2},\quad \left(\pa_r - \frac 2r\right) \pa_r Q > 0.\label{eqnumQ}
   \ee
\end{lemma}
\begin{proof} Compute
\bee
 Q' = -8r\frac{r^2 + 10}{(2+r^2)^3},\quad \partial_r^2 Q = 8\frac{3r^4 + 44r^2 - 20}{(2+r^2)^4}, \quad \left(\pa_r - \frac 2r\right) \pa_r Q = 8\frac{5r^4 + 68r^2 + 20}{(2+r^2)^4}.
\eee
This implies the positivity of $(\pa_r - \frac 2r)\pa_r Q$. The upper bounds of $Q$, $Q''$ follow a computation of maximum for $Q r^2$ and $\partial_r^2 Q \cdot (2+r^2)^2$.  
\end{proof}

In addition, we will prove the quantitative interpolation estimate to bound the nonlocal term, thanks to the explicit form of $\Delta_l^{-1}$ given by \eqref{eqdefDeltal-1}.

\begin{lemma} For $l \ge 2$, $-l \le m \le l$ and $R > 0$, let $f(x) = f_{lm}(r) Y_{lm}(\Omega) \in H^\infty(\RR^3) \cap \calH_{(l)}$, we have 
\be
\left\| \frac{\Delta_l^{-1} f_{lm} }{r(2+r^2)}\right\|_{L^2}^2 \le \frac{4\a(R)}{(2l+1)^2(2l-3)}\left\| \frac{f_{lm}}{r} \right\|_{L^2}^2 + \frac{4\beta(R) }{(2l+1)^2 (2l-1)} \| f_{lm} \|_{L^2}^2,
\label{eqnonlocalestintp}
\ee
where 
\be
\begin{split}
 \a(R)& = \int_0^R \frac{r^3}{(2+r^2)^2} dr = \left(\frac{1}{R^2 + 2} + \frac{\log(R^2 + 2)}{2}\right) - \left(\frac 12 + \frac{\log 2}{2}\right),\\
 \beta(R) &= \int_R^\infty \frac{r}{(2+r^2)^2} dr =  \frac{1}{2(R^2 + 2)}.
 \end{split}\label{eqdefalphabeta}
\ee
\end{lemma}
\begin{proof}
    Using \eqref{eqdefDeltal-1} together with Cauchy-Schwarz inequality, we have
    \bee
    | \Delta_l^{-1}f_{lm}(r)| &\le& \frac{1}{2l+1} \Bigg[ r^{-(l+1)} \left(\int_0^r |f_{lm}|^2 ds\right)^\frac 12 \left( \int_0^r s^{2(2+l)} ds \right)^\frac 12\\
    && + r^{l} \left(\int_r^\infty |f_{lm}|^2 ds\right)^\frac 12 \left( \int_r^\infty s^{2(1-l)} ds \right)^\frac 12 \Bigg] \\
    &\le& \frac{1}{2l+1} \left( \int_0^\infty |f_{lm}|^2 ds \right)^\frac 12\cdot  r^\frac 32\left( (5+2l)^{-\frac 12} + (2l-3)^{-\frac 12} \right)  \\
    &\le& \frac{2r^{\frac 32}}{(2l+1)(2l-3)^\frac 12} \cdot (4\pi)^{-\frac 12} 
\left\| \frac{f_{lm}}{r} \right\|_{L^2(\RR^3)},
    \eee
    where $l\ge 2$ guarantees the integrability. Similarly, for $l \ge 1$, we can obtain that
    \bee
    |\Delta_l^{-1} f_{lm}(r)| &\le& \frac{2r^\frac 12}{(2l+1)(2l-1)^\frac 12} \cdot (4\pi)^{-\frac 12} \| f_{lm} \|_{L^2}^2.
    \eee
    Therefore, \eqref{eqnonlocalestintp} follows an application of the above two pointwise estimates on $r \in [0, R]$ and $r \in [R, \infty)$ respectively.
\end{proof}

We finally prove Proposition \ref{propmodstabl3} with the preliminaries above.
\begin{proof}[Proof of Proposition \ref{propmodstabl3}]
  Since $\calH_{(l)} = \bigoplus_{m=-l}^{l} L^2(\RR^2,r^2dr) \bigotimes \calY_{lm}$, 
  using the orthogonal condition $(Y_{lm},Y_{l'm'})_{L^2(\mathbb{S}^2)} = \delta_{l=l',m=m'}$ together with Lemma \ref{thmLlresonHl}, it suffices to prove \eqref{Llposivwith3} with $f(x) = f_{lm}(r) Y_{lm}(\Omega)$.
  In fact, plugging the numerology \eqref{eqnumQ} and Hardy's inequality $\| r^{-1}f \|_{L^2(\RR^3)} \le 2\| \nabla f \|_{L^2(\RR^3)}$ (see \cite[Theorem $1.72$]{MR2768550}) into \eqref{eqquacalLl}, we find 
  \bee
  4 \pi \Re (\calL f, f)_{L^2} \ge \left(l(l+1) - \frac{13}{2} \right)\left\| \frac{f_{lm}}{r} \right\|_{L^2}^2 + \frac 14 \| f_{lm}\|_{L^2}^2 - \frac{68 l(l+1)}{3}\left\| \frac{\Delta_l^{-1} f_{lm} }{r(2+r^2)}\right\|_{L^2}^2.
  \eee
  Now fix $R = 4$, we evaluate the values in \eqref{eqdefalphabeta} that
  \[ \a(4) < \frac 23,\quad \beta(4) = \frac{1}{36}. \]
  Then for $l \ge 3$, from \eqref{eqquacalLl}, \eqref{eqnonlocalestintp} and \eqref{eqdefalphabeta}, we obtain that
  \bee
   4 \pi \Re (\calL f, f)_{L^2} &\ge& \left( l(l+1) - \frac{13}{2} - \frac{68l(l+1)}{3}\cdot\frac{4 \cdot \frac 23}{(2l+1)^2(2l-3)} \right) \left\| \frac{f_{lm}}{r} \right\|_{L^2}^2 \\
   &+&\left( \frac 14 -\frac{68l(l+1)}{3}\cdot \frac{4 \cdot \frac 1{36}}{(2l+1)^2(2l-1)} \right) \| f_{lm} \|_{L^2}^2 \\
   &\ge& l(l+1)\left(1 - \frac{13}{24} - \frac{68\cdot 4 \cdot \frac 23}{3 \cdot 7^2 \cdot 3} \right)  \left\| \frac{f_{lm}}{r} \right\|_{L^2}^2 + \left( \frac 14 - \frac{68 \cdot 12 \cdot 4 \cdot \frac{1}{36}}{3 \cdot 7^2 \cdot 5} \right) \| f_{lm} \|_{L^2}^2\\
   &=& \frac{499}{10584}  \left\| \frac{f_{lm}}{r} \right\|_{L^2}^2 + \left( \frac 18 + \frac{29}{17640} \right) \| f_{lm} \|_{L^2}^2 \ge \frac 18 \| f_{lm} \|_{L^2}^2 = \frac{\pi}{2} \| f\|_{L^2}^2,
  \eee
  where we exploited the monotonicity of $\frac 1{l(l+1)}$, $\frac 1{(2l+1)^2(2l-3)}$ and $\frac{l(l+1)}{(2l+1)^2(2l-3)}$ for $l \ge 3$. 
\end{proof}

 \subsection{Mode stability for $l = 2$} 
 \label{subsecl2}
 For the case $l=2$, although the non-existence of unstable eigenvalue is still expected, the direct coercivity in $L^2(\RR^3)$ as \eqref{Llposivwith3} seems hard to prove due to worse coercivity of $-\Delta_l$. We perform a refined analysis through partial localization, conjugation by some suitable function and symmetrization to find a positive Schr\"{o}dinger type operator which is bounded from below via GGMT estimate.

Firstly, motivated by the partial-mass localization in $l=0$ case, or rather, the decomposition formula \eqref{eqDeltalDl}, we introduce $D_{l+2}$ and $D_{l+2}^{-1}$ to partially localize the nonlocal operator $\calL_l$ and gain more coercivity. 
\begin{lemma} As for $l \ge 0$ and $\calL_l$ given in \eqref{eqdefcalLl}, we have 
    \be 
    D_{l+2}^{-1} \calL_l D_{l+2} = -\pa_r^2 - \frac 2r \pa_r + \frac{(l+1)(l+2)}{r^2} + \frac 12 (r \pa_r + 1) - D_2^{-1} Q\cdot D_2 - Q + l T_l,
    \label{eqlocalcalLl1}
    \ee
    where $T_l$ is a non-local operator defined by 
    \be 
     T_l f =  D_{l+2}^{-1} \left( V_1 f + V_2  D_{-l}^{-1} f \right),
     \label{Tl: def}
     \ee
    with 
    \be 
V_1 = -\pa_r \left( r^{-1} D_2^{-1} Q\right) = \frac{8r}{(r^2 + 2)^2},\quad  V_2 =  -r^{-1} \pa_r Q = \frac{8(r^2 + 10)}{(r^2 + 2)^3}. \label{eqdefV1V2}
    \ee
\end{lemma}
\begin{proof} We compute the conjugation of each term in $\calL_l$. From the definition of $D_k$ and $D_k^{-1}$ given \eqref{Dkdef} and \eqref{eqdefDk-1}, 
\bee
 D_{l+2}^{-1} \circ \left(-\frac 12 \Lambda \right)D_{l+2} &=& -\frac 12 \Lambda -\frac 12 D_{l+2}^{-1} [\Lambda, D_{l+2}] = -\frac 12 (\Lambda - 1),  \\
 D_{l+2}^{-1} \circ (-Q) D_{l+2} &=& -Q - D_{l+2}^{-1} [Q, D_{l+2}] = -Q + D_{l+2}^{-1}\circ \pa_r Q,
\eee
and 
\bee
 &&D_{l+2}^{-1} \circ \left( -Q - \pa_r \Delta^{-1} Q \cdot \pa_r \right) D_{l+2}f = -D_{l+2}^{-1} \left( D_2 \left( D_2^{-1} Q \cdot D_{l+2} f \right) \right) \\
 &=& -D_2^{-1} Q \cdot D_{l+2} f + D_{l+2}^{-1} \left( \frac lr D_2^{-1} Q \cdot D_{l+2} f \right) \\
 &=& -D_2^{-1} Q \cdot D_{l+2} f + \frac lr D_2^{-1} Q \cdot f - D_{l+2}^{-1}\left( \pa_r \left( \frac lr D_2^{-1} Q \right) f \right) \\
 &=& -D_2^{-1} Q \cdot D_2 f - D_{l+2}^{-1}\left( \pa_r \left( \frac lr D_2^{-1} Q \right) f \right).
\eee

Further invoking \eqref{eqDeltalDl}, we have 
\bee -D_{l+2}^{-1}\Delta_l D_{l+2} &=& -D_{-l} D_{l+2} = -\pa_r^2 - \frac 2r \pa_r + \frac{(l+1)(l+2)}{r^2}, \\
-D_{l+2}^{-1} \circ \left( \pa_r Q \cdot \pa_r\Delta_l^{-1} D_{l+2}\right)  &=& -D_{l+2}^{-1} \circ \left( \pa_r Q \cdot  \pa_r D_{-l}^{-1} \right)\\
&=& -D_{l+2}^{-1} \circ \left(\pa_r Q  + \pa_r Q \cdot \frac lr D_{-l}^{-1}\right).
\eee
Summing up these 5 terms yields \eqref{eqlocalcalLl1}, where $D_{l+2}^{-1} \circ \pa_r Q$ terms are cancelled. 
\end{proof}

Next, as for $D_{l+2}^{-1} \calL_l D_{l+2}$ given above, we further take the conjugation by $r^\alpha$, symmetrize the related operator and estimate the nonlocality, so as to find a suitable Schr\"{o}dinger operator to bound from below. 

\begin{lemma}\label{lemsymestnon}
    For $l \ge 0$ and $\a \in [-l,l+\frac{1}{2})$, we define the operator
     \be 
  \begin{split}
  \tilde \calL_{l,\a} := r^\a D_{l+2}^{-1} \calL_l D_{l+2} r^{-\a} = &-\pa_r^2 - \frac{2 -2\a}{r}\pa_r  + { \frac{\a-\a^2 + (l+1)(l+2)}{r^2}} \\
 &  + \frac 12 \left( r\pa_r + 1 -\a \right) - D_2^{-1} Q D_{2-\a} - Q + l r^{\a} T_l r^{-\a},
   \end{split}
   \label{tildeLlapha}
  \ee
  with domain $C_c^\infty(\RR^+)$. Then for any chosen function $W$ satisfying
  \be W(r) > 0 \quad \text{ with } \quad W(r) \gtrsim \la r \ra^{-\min \{ 2l+2\a -1, 2 \} + \epsilon} \text{ for some } \epsilon>0,
  \label{eqcondW} \ee
  we have
  \be
  \Re (\tilde \calL_{l,\a} f,f)_{L^2(\RR^+)} \ge (\calH_{l,\a; W} f,f)_{L^2(\RR^+)}, \quad \forall f \; \in C_c^\infty(\RR^+),
  \label{eqlowbddSchop}
  \ee
  where $\calH_{l,\a; W}$ is a Schr\"{o}dinger type operator defined by
  \be \calH_{l, \a; W} = - \pa_r^2 +  \frac{L_{l, \a}}{r^2} + \frac{1 - 2\a}{4} + \frac 12 D_{2\a - 4} D_2^{-1} Q - Q - l \mu_{l, \a}[W^{-1}] W,
 \label{loweroper2}
 \ee  
 with $L_{l, \a} = { -(\a-1)^2} + (l+1)(l+2)$, and the linear functional $\mu_{l,\a}$  given by \eqref{eqdefmula}. 
\end{lemma}

  \begin{proof}
   With the definition of $D_k$ and $D_k^{-1}$ given \eqref{Dkdef} and \eqref{eqdefDk-1}, 
  \bee
&   r^\a D_k r^{-\a} = D_{k-\a}, \quad r^\a \pa_r^2 r^{-\a} = \pa_r^2 - \frac{2\a}{r} \pa_r + \frac{\a (\a+1)}{r^2} \quad  \forall k, \a \in \RR, \\
& {  r^{\alpha} D_k^{-1} r^{-\alpha} = D_{k-\alpha}^{-1}, \quad \text{ where } \max\{k,k-\alpha \} \le 0 \text{ or } \min \{k,k-\alpha \} >0.}
  \eee
As for the expression given above, we can obtain the expression of $\tilde \calL_{l,\a}$ by 
\begin{align*}
\tilde \calL_{l,\a} &=  r^{\a} \left( -\partial_r^2 -\frac{2}{r} \partial_r  + \frac{(l+1)(l+2)}{r^2} + \frac{1}{2} (r \partial_r +1) -D_2^{-1} Q D_2 -Q + lT_l \right) r^{-\a} f \\
& = -\partial_r^2 f  + \frac{2\a}{r} \partial_r f- \frac{\a (\a +1)}{r^2} f  -2 \left( -\frac{\a}{r^2} f + \frac{1}{r} \partial_r f \right) + \frac{(l+1)(l+2)}{r^2} f  \\
& \quad + \frac{1}{2} (-\alpha f + r \partial_r f +f) - D_2^{-1} Q \cdot  D_{2-\alpha} f  - Qf + l r^{\alpha} T_l r^{-\a} f \\
& = -\partial_r^2 f  - \frac{2-2\a}{r} \partial_r f + \frac{-\a^2+ \a + (l+1)(l+2)}{r^2} f  + \frac{1}{2} (r\partial_r +1 -\a) f  \\
& \qquad -D_2^{-1} Q \cdot D_{2-\a} f - Qf + l r^\a T_l r^{-\a} f,
\end{align*}
and the nonlocal term should be
\begin{align*}
r^\a T_l r^{-\a} f
& = r^\a \circ D_{l+2}^{-1} \circ (V_1 + V_2 \circ D_{-l}^{-1}) r^{-\a}f \\
& = r^\a \circ D_{l+2}^{-1}  \circ r^{-\a} \circ  r^\a \circ (V_1 + V_2 \circ D_{-l}^{-1}) r^{-\a}f \\
& = D_{l+2-\a}^{-1} (V_1 f +V_2r^\a D_{-l}^{-1} r^{-\a} f)  \\
& = D_{l+2-\a}^{-1} V_1 f + D_{l+2-\a}^{-1} V_2 D_{-l-\a}^{-1} f,
\end{align*}
where we exploited $l+2-\a > 0$ and $-l-\a \le 0$. 

Now we compute $\calH_{l, \a; W}$ to satisfy \eqref{eqlowbddSchop}. All terms except the last one in \eqref{loweroper2} come from symmetrization of the local terms of $\tilde \calL_{l, \a}$ \eqref{tildeLlapha}, so we are left to determine $\mu_{l, \a}$ such that for all $f \in C_c^\infty(\RR^+)$, 
\be \Re \left((D_{l+2-\a}^{-1} V_1 + D_{l+2-\a}^{-1} V_2 D_{-l-\a}^{-1}) f,f\right)_{L^2(\RR^+)} \ge -(\mu_{l,\a}[W^{-1}] W f,f)_{L^2(\RR^+)}.  \label{eql2symW-1} \ee

We begin by showing the coercivity of the first term in \eqref{eql2symW-1}. Note that as an operator on $L^2_{\RR^+}$, 
\[ (D_k^{-1})^* = -D_{-k}^{-1},\quad \forall \,\, k \neq 0, \] 
% So if $-l \le \a < l+2$, we have
% \bee
%   r^\a T_l r^{-\a} = D_{  l+2-\a}^{-1} \left( V_1 f + V_2 D_{  -l - \a}^{-1} f \right)
% \eee
% Noticing that 
we obtain that
\bee
 D_{l+2-\alpha}^{-1} \circ V_1 &=& D_{l+2-\a}^{-1} \circ V_1 \circ D_{-(l+2-\alpha)} \circ D_{-(l+2-\alpha)}^{-1},\\
 \left(D_{l+2-\a}^{-1} \circ V_1 \right)^* &=& D_{l+2-\a}^{-1} \circ (-D_{ l+2-\a}) \circ V_1 \circ D_{-(l+2-\alpha)}^{-1},
\eee
and it yields that
\bee
 && \frac 12 \left( D_{l+2-\a}^{-1} \circ V_1 +\left(D_{l+2-\a}^{-1} \circ V_1 \right)^* \right) = -D_{l+2-\a}^{-1} \circ W_{1;{ l-\a}}  \circ D_{-(l+2-\alpha)}^{-1}  \\
 &=& \left( W_{1;{ l-\a}}^\frac 12  \circ D_{l+2-\alpha}^{-1}  \right)^* \left( W_{1;{ l-\a}}^\frac 12  \circ D_{l+2-\alpha}^{-1}  \right),
\eee
where 
\be W_{1;{ l-\a}} = \frac 12 \pa_r V_1 + \frac{ l+2-\a}{r}V_1 = \frac{8{ (l-\a)}+4}{(r^2 + 2)^2} + \frac{32}{(r^2 + 2)^3} > 0.  \label{eqdefW1l+a} \ee
Notice that ${l-\a} > - \frac 12$ guarantees the positivity so as to take square root, and $W_{1;{ l-\a}} \sim \la r \ra^{-4}$.

Next, we control the second term in \eqref{eql2symW-1}. By Cauchy-Schwarz inequality, 
\bee
 && \Re \left(D_{l+2-\a}^{-1} V_2 D_{ -l-\a}^{-1} f, f  \right)_{L^2(\RR^+)} = -\Re \left( W_{1;{ l-\a}}^{-\frac 12} V_2 D_{-{ l-\a}}^{-1} f, W_{1;{ l-\a}}^\frac 12 D_{l+2-\alpha}^{-1} f \right)_{L^2(\RR^+)} \\ 
 &\ge& -\frac 14 \left\|  W_{1;{ l-\a}}^{-\frac 12} V_2 D_{-{ l-\a}}^{-1} f \right\|_{L^2(\RR^+)}^2 - \left\|  W_{1, { l-\a}}^\frac 12 D_{l+2-\alpha}^{-1} f \right\|_{L^2(\RR^+)}^2.
\eee
Recall the definition of $D_{\beta}^{-1}$ with $ \beta = { -l-\a} \le 0$, then again by Cauchy-Schwarz, we have the pointwise bound with any positive potential $W \gtrsim \la r \ra^{2\beta+1+}$ that
\bee
 |D_{\beta}^{-1} f(r)|^2 \le r^{-2\beta} \left( \int_r^\infty f^2 W ds \right) \cdot \left( \int_r^\infty W^{-1} s^{2\beta} ds \right).
\eee
So 
\bee
  \frac 14\left\|  W_{1;{ l-\a}}^{-\frac 12} V_2 D_{{ -l-\a}}^{-1} f \right\|_{L^2(\RR^+)}^2 \le \mu_{l, \a}[W^{-1}] ( W f, f)_{L^2(\RR^+)}^2,
\eee 
where the linear functional is
\bea
  \mu_{l, \a}[W^{-1}] &=& \textcolor{black}{\frac{1}{4}} \int_0^\infty W_{1;{ l-\a}}^{-1} V_2^2 r^{{ 2l+2\a}} \left(\int_r^\infty W^{-1} s^{{ -2l-2\a}} ds\right) dr  \nonumber \\
  &=& \textcolor{black}{\frac{1}{4}} \int_0^\infty W^{-1} s^{{ -2l-2\a}} \left(\int_0^s W_{1;{ l-\a}}^{-1} V_2^2 r^{{ 2l+2\a}} dr \right) ds,
  \label{eqdefmula}
\eea
with $V_2, W_{1;{ l-\a}}$ defined as in \eqref{eqdefV1V2} and \eqref{eqdefW1l+a}. The condition \eqref{eqcondW} ensures the exchange of integration order and  $\mu_{l, \a}[W^{-1} ] < \infty$. Combining these estimates yields \eqref{eql2symW-1} and therefore conclude the proof.
  \end{proof}

In order to obtain the spectral information of Schr\"{o}dinger operator $\calH_{l,\a;W}$ defined in \eqref{loweroper2}, we turn to the following GGMT method, which can describe the number of unstable eigenmodes.

\begin{theorem}[GGMT bound, {\cite[Theorem A.1]{MR4126325}}] \label{thmGGMT} For $l \in [0, \infty)$, suppose the linear operator defined on $C_c^\infty(\RR^+)$ by
\[H_l = -\pa_r^2 + \frac{l(l+1)}{r^2} + V, \]
with real-valued potential $V \in L_{loc}^2(\RR^+)$ is such that the closure of $H_l$ (which we still denote by $H_l$) is self-adjoint on $L^2(\RR^+)$. If the quantity $N_{p,l}(V)$ satisfies
\be
  N_{p, l}(V) := \frac{(p-1)^{p-1} \Gamma(2p)}{p^p \Gamma(p)^2} (2l+1)^{-(2p-1)} \int_0^\infty r^{2p-1} |V_-(r)|^p dr < 1, \label{eqNplGGMT}
\ee
where $V_- = \min \{ 0, V \}$, then there is no non-positive eigenvalue of the self-adjoint operator $H_l$ on $L^2(\RR^+)$.
\end{theorem}

Finally, we choose an appropriate function $W$ satisfying \eqref{eqcondW} and $\alpha$ such that the Schr\"{o}dinger operator $\calH_{l,\a;W}$ given in \eqref{loweroper2} is positive so that it yields the coercivity of $\calL_{l,\alpha}$ with $l=2$.
\begin{proposition}
\label{propL2coer}
      Let $l=2$ and $\a =0.2$, then there exists a uniform constant $\epsilon_0 > 0$ such that the operator $\tilde\calL_l$ defined in \eqref{tildeLlapha} satisfies
      \be
      \Re (\tilde \calL_{l,\alpha} g,g)_{L^2(\RR^+)}
      \ge \epsilon_0 \| g \|_{L^2(\RR^+)}^2, \quad \forall g \in C_c^\infty(\RR^+).
      \label{coercivityLl2}
      \ee
  \end{proposition}
  \begin{proof} From \eqref{eqlowbddSchop}, it suffices to prove the lower bound of $\Re (\calH_{l, \a; W} g, g)$ for $g \in C^\infty_c(\RR^+)$ with the operator from \eqref{loweroper2}. We introduce $\theta \in [0, 1]$ to rewrite $\calH_{l,\alpha;W}$ as
      \begin{align*}
\calH_{l,\alpha;W} = -\partial_r^2 + \frac{(1-\theta) L_{l,\alpha}}{r^2} + U_{l,\alpha,\theta;W}
          = -\partial_r^2 + \frac{l_{\text{eff}}(l_{\text{eff}}+1)}{r^2} +  U_{l,\alpha,\theta;W},
      \end{align*}
      where
      \[
      l_{\text{eff}} = \left( \frac{1}{4} + (1-\theta) L_{l,\alpha} \right)^\frac{1}{2} - \frac{1}{2},
      \]
      and
      \[
      U_{l,\alpha,\theta;W} = \frac{\theta L_{l,\alpha}}{r^2} + \frac{1 - 2\a}{4} + \frac 12 D_{2\a - 4} D_2^{-1} Q - Q - l \mu_{l, \a}[W^{-1}] W.
      \]
      With $l = 2$, $\a = 0.2$ and
\be
 p = 4,\quad \theta = 0.5, \quad
 W(r) = (0.01 + r^2)^{-1.2} + 0.02,
\ee
one can compute the integrals \eqref{eqdefmula} and \eqref{eqNplGGMT}  numerically 
\be \mu_{l,\a}[W^{-1}] \approx 1.9137,\quad N_{p, l_{\rm eff}}(U_{l, \a, \theta; W}) \approx 0.8687 < 1. \label{eqNpleff} \ee
Now that $\min \{ U_{l, \a, \theta; W}, 0\} \in C^0_c(\RR^+)$ and $(1-\theta)L(l, \a) = 0.5\cdot L(2,0.2) = 0.5 \cdot(12 - 0.8^2) > \frac 34$, standard theory (see e.g. \cite[Theorem X.7, X.8, X.10]{zbMATH03483022}) implies that $\calH_{l, \a; W}$ is essentially self-adjoint on $L^2(\RR^+)$. Therefore the GGMT bound (Theorem \ref{thmGGMT}) and the numerics \eqref{eqNpleff} indicates the non-existence of non-negative eigenvalue. Moreover, noticing that 
$$\min \left\{ U_{l, \a, \theta; W} - \frac{1-2\a}{4} + l\mu_{l, \a}[W^{-1}] W(\infty) + \epsilon, 0\right\} \in C^0_c(\RR^+),\quad \forall \epsilon > 0,$$ 
we also have $\sigma_{ess}(\calH_{l,\a;W}) \subset [\frac{1-2\a}{4}-2 \mu_{l, \a}[W^{-1}] W(\infty),\infty) \subset  [0.07, \infty)$. These facts yield the existence of $\epsilon_0 > 0$ such that  \begin{align*}
        \Re \left( \tilde \calL_{l,\alpha} g, g \right)_{L^2(\RR^+)} 
        \ge \left( \calH_{l,\alpha;W} g, g \right)_{L^2(\RR^+)}  
      \ge \epsilon_0 \| g \|_{L^2(\RR^+)}^2, \; 
      \quad \forall \; g \in C_c^\infty(\RR^+).
      \end{align*}
  \end{proof}

 \subsection{Mode stability for $l = 1$.}\label{sec24} This section is devoted to the behavior of $\calL_1$, in which case the methods introduced in Section \ref{subsecl3} and \ref{subsecl2} fail due to the existence of unstable eigenpair $(\calL_1 + \frac 12)\pa_r Q = 0$ \eqref{eqexplicitunstablemode} and the weak coercivity from angular momentum. 
 
 Here, motivated by \cite{oseenli}, we can construct a formal wave operator such that it can transform $\calL_1$ into a local operator and eliminate the unstable eigenpair. 

 \begin{proposition} \label{propwaveop2}
     We define a linear operator $T$ as
    \be  
     T := I - \frac{\pa_r Q}{D_3^{-1} \pa_r Q}D_3^{-1},
     \label{waveop2}
    \ee
    and a nonlocal differential operator $\tilde \calL_1$ by
     \be
    \tilde \calL_1  := -\Delta_1 + \frac 12 \Lambda - D_2^{-1} Q \cdot \pa_r - 2Q - 2\pa_r \left(\frac{\pa_r Q}{D_3^{-1}\pa_r Q} \right), 
    \label{localLl1}
    \ee
    where $D_3^{-1}$ is given by \eqref{Dkdef} and $\Delta_1$ is given by \eqref{eqdefDeltal-1}. Then we have
     \be
    T \calL_1 f = \tilde \calL_1 T f, \quad \text{ for any } f \in \calD_1,
    \label{eqwaveopcom}
    \ee
    with
    \begin{align}
    \calD_1 := \Big\{ f(r)\in L^2(\RR^+; r^2 dr): \;   & f(r) Y_{1,m}(\Omega) \in H^\infty(\RR^3) \text{ and }  \notag \\
    & \calL(f(r) Y_{1,m}(\Omega)) \in H^\infty(\RR^3) \text{ for some $-1 \le m \le 1$}.  
    \Big\}
    \label{domL1}
    \end{align}
\end{proposition}

Before we continue to give the proof of Proposition \ref{propwaveop2}, we give some remarks about the robustness of the wave operator method.
\begin{remark}[Wave operator method]\label{rmkwom}
    Our choice of $T$ is motivated by the wave operator method in \cite{oseenli}, where they considered a certain nonlocal differential operator $\mathscr{L}$ in the first spherical class, with nonlocality from Biot-Savart law $\nabla^\perp \Delta^{-1}$ and a kernel $\mathfrak{g} \in L^2(\RR^+)$. The authors define the wave operator as 
    \[ \mathscr{T} := I - \mathscr{S}, \quad \mathscr{S} := \frac{(\cdot, \mathfrak{g})_{L^2([0, r])}}{(\mathfrak{g}, \mathfrak{g})_{L^2([0, r])}}\mathfrak{g}.  \]
    Then it satisfies $\mathscr{T} \mathfrak{g}= 0$, the wave operator property
     \be  \mathscr{T} \mathscr{T}^* = I,\quad  \mathscr{T} \mathscr{T}^* = I - \frac{(\cdot, \mathfrak{g})_{L^2([0,\infty))}}{(\mathfrak{g}, \mathfrak{g})_{L^2([0, \infty))}}\mathfrak{g}, \label{eqwaveopprop} \ee
    and more importantly $\mathscr{T} \mathscr{L} = \tilde{\mathscr{L}} \mathscr{T}$, where $\tilde{\mathscr{L}}$ is a local differential operator with better coercivity. 
    
    In our case, we generalize this construction to the weighted space $L^2_\omega$ and our operator $T$ \eqref{waveop2} can be rewritten as 
    \[ T = T_\omega := I - S_\omega, \quad S_\omega := \frac{(\cdot, \pa_r Q)_{L^2_\omega([0, r])}}{(\pa_r Q, \pa_r Q)_{L^2_\omega([0, r])}}\pa_r Q, \quad  \omega = -\frac{r^3}{\pa_r Q}. \]
    Here $L^2_\omega$ is well-defined thanks to that $\pa_r Q < 0$ for any $r > 0$, and the weight is selected to uniformize nonlocal terms in $T_\omega \calL_1$. 
    
    We also stress that our $T_\omega$ is not a wave operator in the sense of \eqref{eqwaveopprop} since $\pa_r Q \notin L^2_\omega([0,\infty))$. Nevertheless, $T$ is still well-defined, $T (\pa_r Q) = 0$ and the commutator property \eqref{eqwaveopcom} holds, which indicates that $\sigma(\tilde \calL_1)$ includes $\sigma(\calL_1)$ except the unstable eigenpair $(-\frac 12, \pa_r Q)$. 
\end{remark}

\begin{remark}[Application for other profiles] \label{rmkprof}
    From the proof, this proposition holds for any radial $Q$ solving the profile equation \eqref{eqprofileQ} as long as 
    \be  D_3^{-1} \pa_r Q(r) \neq 0,\quad \forall \,\,r > 0,\label{eqnonvanishcond} \ee
    to ensure $T$ and $\tilde \calL_1$ well-defined as \eqref{waveop2} and \eqref{localLl1}. For the excited self-similar profiles found in \cite{Brenner_Constantin_Leo_Schenkel_Venkataramani_steady_state_99, MR1651769}, we could numerically check this non-vanishing condition to hold for the first three profiles, although an analytical proof is still missing. 
\end{remark}

\begin{remark}[Application in radial case]
\label{rmkradial}
    In the radial case, the corresponding linear operator
    \[ T_0 := I - \frac{g_0}{D_2^{-1} g_0} D_2^{-1},\quad g_0 = \Lambda Q, \]
   can also localize $\calL_0$ and remove the unstable eigenmodes, offering a different perspective to the proof of \cite[Proposition $4.3$]{MR4685953}. Indeed, it is equivalent to the combination of the partial mass coordinate for localization \eqref{eqlocalcalLl1} and supersymmetry method for removing unstable mode (\cite[Proposition $4.3$]{MR4685953}) by noticing that 
   \[ T_0 = \left( \pa_r - \frac{(D_2^{-1} g_0)'}{D_2^{-1} g_0} \right) \circ D_2^{-1}. \]
    \end{remark}

Next, we are devoted to the proof of the main result Proposition \ref{propwaveop2}.
\begin{proof}[Proof of Proposition \ref{propwaveop2}]
First of all, for any $f \in \calD_1$ (see \eqref{domL1} for the definition), it is easy to verify that $T \calL_1 f$ and $\tilde \calL_1 Tf$ are both pointwisely well-defined on $\RR^+$, hence it suffices to verify the commutator property \eqref{eqwaveopcom}. For notational simplicity, we define 
\be g := \pa_r Q,\quad  G := r^3 D_3^{-1} g = Qr^3 -  3r^2 D_2^{-1} Q,  \label{eqGdef} \ee
then $T$ can be rewritten as 
\be T = I - \frac{g}{G}\int_0^r (\cdot) s^3 ds.  \label{eqTdef2} 
\ee

    \underline{Step 1. Computing $T \calL_1$.} Applying \eqref{eqDeltalintop1}, we decompose 
    \bea
T \calL_1 f &=& T \left[ (-\Delta_1 + \frac 12 \Lambda)f + (-2Q - D_2^{-1} Q\cdot \pa_r)f - \frac 23 g \cdot D_3^{-1} f - \frac 13 g \cdot D_0^{-1} f \right] \nonumber\\
&:=& \calI_1 + \calI_2 -\frac 23 \calI_3 - \frac 13 \calI_4. \label{eqcomputeTcalL1}
\eea
We compute each term. For $\calI_1$,
\bee
\calI_1 
&=& (-\Delta_1 + \frac 12 \Lambda) f + \frac{g}{G}\int_0^r \left(\pa_s^2 + (\frac 2s - \frac s2 )\pa_s  - \frac{2}{s^2}- 1\right)f s^3 ds \\
&=&( -\Delta_1 + \frac 12 \Lambda) f + \frac{g}{G} \left[ \pa_r f \cdot r^3 - \pa_r (r^3) f + \left(\frac 2r - \frac r2\right) r^3 f \right]  \\
&+& 
\frac{g}{G} \int_0^r f \cdot \left[ \pa_s^2 (s^3) - \pa_s \left( (\frac 2s - \frac s2) s^3 \right) - (\frac{2}{s^2} + 1) s^3  \right] ds \\
&=& \left( -\Delta_1 + \frac 12 \Lambda + \frac {g}{G} \left(r^3\pa_r - r^2 - \frac 12r^4 \right)\right)f + \frac gG \int_0^r fs^3 ds. 
\eee

Similarly, for $\calI_2$, 
  \bee \calI_2
  &=& \left(- D_2^{-1} Q \cdot \pa_r - 2Q + \frac{g}{G}r^3 D_2^{-1} Q\right) f + \frac gG \int_0^r f \left( -\pa_s (D_2^{-1} Q s^3) + 2Q s^3 \right)ds \\
  &=& \left(- D_2^{-1} Q \cdot \pa_r - 2Q + \frac{g}{G}r^3 D_2^{-1} Q\right) f + \frac gG \int_0^r f \cdot \left( Q s^3 - s^2 D_2^{-1} Q  \right)ds.
  \eee
  
  Next for $\calI_3$, 
  \bee
\calI_3
&=& \frac{g}{r^{3}} \int_0^r fs^3 ds - \frac{g}{G} \int_0^r g(s) \int_0^s f(t) t^3 dt ds \\
&=& \frac{g}{r^{3}} \int_0^r fs^3 ds - \frac{g}{G}Q \int_0^r f(t) t^3 dt + \frac gG \int_0^r f(t) t^3 Q(t) dt,
\eee
where we exchange the integration order and used $g = \pa_r Q$ so that $\int_t^r g(s) ds = Q(r) - Q(t)$.

 For $\calI_4$, with the definition of $G$ \eqref{eqGdef}, we obtain
\bee
\calI_4 
&=& - g\int_r^\infty f(s) ds + \frac{g}{G} \int_0^r g(s) s^3 \int_s^\infty f(t) dt ds \\
&=& - g\int_r^\infty f(s) ds + \frac{g}{G} \left(\int_0^r g(s) s^3 \int_s^r f(t) dt ds + \int_0^r g(s) s^3 ds \cdot \int_r^\infty f(t) dt \right) \\
&=& \frac gG \int_0^r f(t) G(t) dt = \frac gG \int_0^r f \cdot \left( Q s^3 - 3s^2 D_2^{-1} Q \right) ds.
\eee

Then plugging these components into \eqref{eqcomputeTcalL1} yields
\be
  T \calL_1 f = -\pa_r^2 f+ A_0 (r)\pa_r f+ B_0(r)f + C_0(r) \int_0^r fs^3 ds, \label{eqTcalL1compute1} \ee
  where 
  \be
  \begin{cases}
        A_0(r) = -\frac 2r + \frac 12 r + \frac gG r^3 - D_2^{-1} Q,   \\
    B_0(r) = \frac {2}{r^2} + 1 - 2Q + \frac gG \left(-r^2 - \frac 12 r^4 + D_2^{-1} Q r^3 \right),  \\
    C_0(r) = \frac gG\left(1 - \frac 23 r^{-3} G + \frac 23 Q \right) = \frac gG\left(1 + \frac 2r D_2^{-1} Q \right).
  \end{cases}
  \label{eqTcalL1compute2} 
\ee

\mbox{}

\underline{Step 2. Verification of \eqref{eqwaveopcom}.} Rewrite $\tilde \calL_1$ \eqref{localLl1} as 
\[
\tilde \calL_1 = -\partial_r^2 + A(r) \partial_r  + B(r),
\]
with 
\be
\begin{cases}
     A(r) =  -\frac{2}{r} + \frac{1}{2} r -D_2^{-1} Q,  \\
    B(r) = \frac{2}{r^2} +1 -2Q - 2\pa_r \left( \frac{g}{G} \right)
     = \frac{2}{r^2} +1 -2Q - 2 \partial_r \left( \frac{g}{G} \right) r^3 -6r^2 \frac{g}{G}.
\end{cases}
  \label{eqtildecalL1AB}
\ee
Then a direct differentiation of $Tf$ \eqref{eqTdef2} implies
\bea
  \tilde \calL_1 Tf &=& -\partial_r^2 f + \frac g G \pa_r (r^3 f) + 2 \pa_r \left( \frac{g}{G} \right) r^3 f +  \partial_r^2 \left( 
 \frac{g}{G}\right) \int_0^r f s^3 ds\nonumber \\
 &+& A \left( \partial_r f -  \frac{g}{G} r^3 f -  \partial_r \left( 
 \frac{g}{G}\right) \int_0^r f s^3 ds  \right) + B \left( f - \frac gG \int_0^r fs^3 ds \right) \nonumber \\
 &=& -\pa_r^2 f  + \left( r^3 \frac{g}{G} + A \right) \partial_r f 
 + \left( 2 \partial_r \left( \frac{g}{G}  \right) r^3 + 3r^2 \frac{g}{G} - A \frac{g}{G} r^3 + B \right) f  \nonumber \\
 &+& \left( \partial_r^2 \left( \frac{g}{G}  \right) - A \partial_r \left( \frac{g}{G} \right) - B \frac{g}{G} \right) \int_0^r f s^3ds. \label{eqtildecalL1T}
\eea
Now it suffices to check the identification of coefficient functions from \eqref{eqtildecalL1T} and \eqref{eqTcalL1compute1}. The forms of $A_0, B_0$ \eqref{eqTcalL1compute2} and $A, B$ \eqref{eqtildecalL1AB} easily implies 
\bee
  r^3 \frac gG + A = A_0,\qquad  2 \partial_r \left( \frac{g}{G}  \right) r^3 + 3r^2 \frac{g}{G} - A \frac{g}{G} r^3 + B = B_0.
\eee
Therefore, the commutator formula \eqref{eqwaveopcom} is reduced to verifying
\bee
   C_0(r) = \partial_r^2 \left( \frac{g}{G}  \right) - A \partial_r \left( \frac{g}{G} \right) - B \frac{g}{G}, 
\eee
with $C_0(r)$ determined in \eqref{eqTcalL1compute2}, namely 
\be
 \left( -\partial_r^2  + A\pa_r + B + 1 + \frac 2r D_2^{-1} Q \right)\left( \frac{g}{G}  \right) = 0.
\label{eqlastident}
\ee
We will be dedicated to proving \eqref{eqlastident}, and write $g = \pa_r Q$ to use the profile equation. 

As preparation, we recall the profile equation \eqref{eqprofileQ} and derive the following identities by taking $\pa_r$ and $D_2^{-1}$ respectively
\bea
  \left(\calL_1 + \frac 12 \right) (\pa_r Q)  &=& 0, \label{eqequationofQ1} \\
  -\pa_r Q + \frac{r}{2} Q - \frac 12 D_2^{-1} Q - Q D_2^{-1} Q &=& 0. \label{eqequationofQ2}
\eea
Note that \eqref{eqequationofQ1} is the eigen equation \eqref{eqexplicitunstablemode}, and for the identity \eqref{eqequationofQ2}, we have exploited $D_2(Q D_2^{-1} Q) = Q^2 + \na Q \cdot \na \Dein Q$. Nex, from $\pa_r G = \pa_r Q \cdot r^3$, we compute 
\bee
  \pa_r \left( \frac{\pa_r Q}{G} \right) &=& \frac{\pa_r^2 Q}G - \frac{(\pa_r Q)^2 r^3}{G^2}, \\
  \pa_r^2 \left( \frac{\pa_r Q}{G} \right) &=& \frac{\pa_r^3 Q}G - 3\frac{\pa_r^2 Q \cdot \pa_r Q r^3}{G^2} - 3\frac{(\pa_r Q)^2 r^2}{G^2} + 2\frac{(\pa_r Q)^3 r^6}{G^3}.
\eee

Therefore 
\bea
  &&\left(-\pa_r^2 + A \pa_r + B\right)\left(\frac{\pa_r Q}{G}\right) \nonumber \\
  &=& - \frac{\pa_r^3 Q}G + 3\frac{\pa_r^2 Q \cdot \pa_r Q r^3}{G^2} + 3\frac{(\pa_r Q)^2 r^2}{G^2} - 2\frac{(\pa_r Q)^3 r^6}{G^3} \nonumber \\
  &+& \left( -\frac{2}{r} + \frac{1}{2} r -D_2^{-1} Q\right)\left( \frac{\pa_r^2 Q}G - \frac{(\pa_r Q)^2 r^3}{G^2}\right)\nonumber  \\
  &+& \left(  \frac{2}{r^2} +1 -2Q \right) \frac{\pa_r Q}{G} - 2 \left( \frac{\pa_r^2 Q}G - \frac{(\pa_r Q)^2 r^3}{G^2}\right) r^3 \frac{\pa_r Q}{G} - 6r^2 \left(\frac{\pa_r Q}{G}\right)^2 \nonumber \\
  &=& \frac 1G \cdot \left( - \pa_r^2 +\left( -\frac{2}{r} + \frac{1}{2} r -D_2^{-1} Q\right)\pa_r +  \frac{2}{r^2} +1 -2Q  \right)(\pa_r Q) \label{eqABcond3term1} \\
  &+& \frac{\pa_r Q\cdot r^3}{G^{2}}  \cdot \left[ \pa_r^2 Q  + \left( -\frac 1r - \frac{r}{2} + D_2^{-1} Q \right) \pa_r Q \right].\label{eqABcond3term2} 
\eea
Applying \eqref{eqequationofQ1} for \eqref{eqABcond3term1} and \eqref{eqprofileQ} for \eqref{eqABcond3term2} implies
\bee
\left(-\pa_r^2 + A \pa_r + B\right)\left(\frac{\pa_r Q}{G}\right) = \frac{\pa_r Q}{G} \cdot \left( \pa_r \Delta_1^{-1} \pa_r Q - \frac 12 \right) + \frac{\pa_r Q \cdot r^3}{G^2} \cdot \left( -\frac 3r \pa_r Q + Q - Q^2 \right).
\eee
Recall that \eqref{eqDeltalintop1} leads to
\bee  \pa_r \Delta_1^{-1} \pa_r Q  = \frac 23 D_3^{-1} (\pa_r Q) + \frac 13 D_0^{-1} (\pa_r Q) 
= \frac 23 Q - \frac 2r D_2^{-1} Q + \frac 13 Q = Q - \frac 2r D_2^{-1} Q. \eee
Hence, we have
\bee
  &&\left(\frac{\pa_r Q\cdot r^3}{G^2}\right)^{-1} \cdot \left[\left(-\pa_r^2 + A \pa_r + B + 1 + \frac 2r D_2^{-1} Q\right)\left(\frac{\pa_r Q}{G}\right)  \right] \\
  &=& \frac{G}{r^3} \cdot \left( \frac 12 + Q \right) + \left( -\frac 3r \pa_r Q + Q - Q^2 \right)\\
  &=& \left( Q - \frac 3r D_2^{-1}Q \right) \cdot  \left( \frac 12 + Q \right)  -\frac 3r \pa_r Q + Q - Q^2 \\
  &=& \frac 3r \left( -\pa_r Q + \frac{r}{2} Q - \frac 12 D_2^{-1} Q - Q D_2^{-1} Q \right) = 0,
\eee
where we used the definition of $G$ from \eqref{eqGdef} and the identity \eqref{eqequationofQ2}. This is equivalent to \eqref{eqlastident} and thus we have concluded the proof. 
\end{proof}

As for the local operator $\tilde \calL_1$ obtained in Proposition \ref{propwaveop2}, it can be further conjugated into a symmetric local operator, making its positivity much easier to verify.

\begin{proposition}
\label{propcoerl1}
    We define a differential operator $\tilde \calL_1'$ with domain $C_c^\infty(\RR^+)$ by
    \be
    \tilde \calL_1' : = -\partial_r^2 + \frac{12}{r^2} + \frac{r^2}{16}- \frac{8}{2+r^2} - \frac{3}{4},
    \label{L1ttilde}
    \ee
    which is a symmetric operator on $L^2(\RR^+)$. Then $\tilde \calL_1'$ satisfies the following properties:

    \noindent $(1)$ $\tilde \calL_1'$ is the conjugation of $\tilde \calL_1$ in the following sense:
    \be
    U_1^{-1} \tilde \calL_1 U_1 g = \tilde \calL_1' g, \quad \forall\;  g \in C^2(\RR^+),
    \label{conjugationtildeL1}
    \ee
    with 
    \be
    U_1(r)
    = e^{\frac{1}{2} \int A(r) dr} 
    = \frac{e^{\frac{1}{8}r^2}}{r(2+r^2)}.
    \label{U1def}
    \ee
    where $A(r)$ comes from \eqref{eqtildecalL1AB}.

    \noindent $(2)$ $\tilde \calL_1'$ is essentially self-adjoint on $L^2(\RR^+)$. In particular, its self-adjoint closure (still denoted by $\tilde \calL_1'$ to simplify notation) has domain
    \be
    D(\tilde \calL_1' ) =  \left\{ u \in L^2(\RR^+): \tilde \calL_1' u \in L^2(\RR^+) \right\},
    \label{domainofclosure}
    \ee
    whose spectrum satisfies
    \be
    \sigma(\tilde \calL_1' ) \subset \Big[\frac{2}{5},\infty\Big).
    \label{spectrumL1tt}
    \ee
\end{proposition}

\begin{proof}[Proof of Proposition \ref{propcoerl1}] First of all, 
% for all $f,g \in C_c^\infty(\RR^+)$, by integration by parts, it is easy to check that
% \[
% (\tilde \calL_1' f,g)_{L^2(\RR^+)} = (f, \tilde \calL_1' g)_{L^2(\RR^+)},
% \]
% which implies that
$\tilde \calL_1'$ is clearly symmetric on $L^2(\RR^+)$ as a Schr\"odinger operator, and $(1)$ is easily verified with simple calculations. For (2), the essential self-adjointness follows from standard theory (see e.g. \cite[Theorem X.7, X.8, X.10]{zbMATH03483022}), similar to Proposition \ref{propL2coer}. And note that
\[
\frac{12}{r^2} + \frac{r^2}{16} - \frac{8}{2+r^2} - \frac{3}{4} \ge \frac{2}{5}, \quad \forall \; r >0,
\]
one immediately concludes \eqref{spectrumL1tt}.
\end{proof}

 \subsection{Proof of Theorem \ref{thm: mode stability main}}
 \label{subseccompleteproofL}
 In this section, we are ready to prove our central result Theorem \ref{thm: mode stability main}.

\begin{proof}[Proof of Theorem \ref{thm: mode stability main}]
Recall the direct decomposition of $L^2(\RR^3)$  \eqref{directdecomofL2}. To prove Theorem \ref{thm: mode stability main}, it suffices to prove that for some $0< \tilde{\delta}  \ll 1$, 
\begin{align}
    &\sigma_{disc} \left( \calL \Big|_{H^k(\RR^3) \cap \calH_{(l)}} \right) \cap \{ z \in \CC: \Re z < \tilde{\delta}  \} = 
    \begin{cases}
        \emptyset, & l \ge 2, \\
        \{ -\frac{1}{2} \}, & l = 1, \\
        \{ -1 \}, & l=0,
    \end{cases}  \label{equnstabspec} 
    % \\
    % &\text{ with } \bigcup_{k \ge 1} \ker (\calL-\l)^{k} = \begin{cases}
    %     \text{span}\{ \partial_{x_1} Q,
    %      \partial_{x_2} Q, \partial_{x_3} Q \},  & \text{ when } \l = -\frac 12, \\
    %      \text{span}\{ \Lambda Q \}, & \text{ when } \l = -1.
    % \end{cases} \label{equnstabspec2}
\end{align}
and \eqref{equnstabspec2}. 

We will proceed for each spherical class by contradictory argument. The core is to check the regularity of the supposed unstable eigenmode, so as to apply the coercivity estimates from Section \ref{subsecl3}-\ref{sec24} to derive a contradiction.

Notice that for any eigenfunction $f \in H^k$, i.e. satisfying $\calL f = zf$, the smoothing estimate of resolvent from \cite[Lemma $2.9$]{ksnsblowup} improves the regularity of eigenfunctions to be 
\[
f \in H^\infty(\RR^3) \cap \calH_{(l)}.
\]

\mbox{}

\noindent \underline{\textbf{Case 1. $l \ge 3$.}} 

\mbox{}

Assume that there exists a nontrivial eigenfunction $f \in H^\infty(\RR^3) \cap \calH_{(l)}$ with $l \ge 3$, such that $\calL f = z f$ with $\Re z < \tilde{\delta} \ll 1 $, then we find that
\begin{align*}
    \Re (\calL f,f)_{L^2(\RR^3)}
    & = \Re z \| f\|_{L^2(\RR^3)}^2 \le \tilde {\delta}  \| f\|_{L^2(\RR^3)}^2,
\end{align*}
which contradicts to Proposition \ref{propmodstabl3} if $\tilde{\delta}  < \frac{1}{8}$.

\mbox{}

\noindent \underline{\textbf{Case 2. $l=2$.}} 

\mbox{}

\underline{\textit{Step $1$.}} Assume that  $0 \not= f \in H^\infty\cap \calH_{(2)}$ is an eigenfunction of $\calL$ with eigenvalue $z$ with $\Re z \le \delta$.
% i.e. $\calL f = zf$. From  \cite[Lemma $2.9$]{ksnsblowup}, we can improve the regularity of eigenfunctions to be $f \in H^\infty(\RR^3) \cap \calH_{(l)}$. In addition,
With the $L^2$ direct decomposition provided in \eqref{directdecomofL2} and invariant action of $\calL$ onto $\calH_{(l)}$ (see Lemma \ref{thmLlresonHl}), it suffices to assume that $f(x) = f_{2,m}(r) Y_{2,m}(\Omega) \in H^\infty(\RR^3) \cap \calH_{(2)}$, where $f_{2,m}$ satisfies
\[
\calL_2 f_{2,m} = z f_{2,m} \quad \text{ with } \quad  f_{2,m}  \in L^2(\RR^+; r^2 dr) \cap L^\infty(\RR^+).
\]

\underline{\textit{Step $2$.}} We claim that 
\be 0 \not\equiv q(r):= r^{0.2} D_4^{-1} f_{2,m} \in L_{r^{-2} + 1}^2(\RR^+) \cap  \dot H_{r^{-1} +r}^{1} (\RR^+) \cap \dot H^2(\RR^+),
\label{q: regularity claim}
\ee
if $0 < \tilde {\delta} < \epsilon_0$ with $\epsilon_0 \ll 1$ given in Proposition \ref{propL2coer}. 

In fact, for $r \in (0,1)$, applying Lemma \ref{lemflmrg} with $l=2$ yields that
\begin{align}
 |\partial_r^i (D_4^{-1} f_{2,m})| \lesssim r^{3-i}, \quad 0 \le i \le 2, \label{regugl2v} 
\end{align}
In addition, when $r \in (1,\infty)$, by Lemma \ref{lemmadecayatinifty}, 
\begin{align*}
& |\partial_r^i (D_4^{-1} f_{2,m})| \lesssim r^{3 \tilde{\delta} -1 -i}, \quad i =0,1, \\
& \Big|\partial_r^2 \big(D_4^{-1} f_{2,m}\big) \Big| 
	\lesssim r^{-3 + 3 \tilde{\delta}} + |\partial_r f_{2,m}|,
\end{align*}
where $\partial_r f_{2,m} \in L^2\left( [1,+\infty); r^{2}dr \right)$ from Lemma \ref{lemflmL2integrable}. Thus combining all of the arguments above yields \eqref{q: regularity claim} and we have concluded the claim.

\underline{\textit{Step $3$.}} From \textit{Step 2}, there exists a sequence $\{ q_n \} \subset C_c^\infty(\RR^+)$ such that
\be
q_n \to q \text{ in } L_{r^{-2} + 1}^2(\RR^+) \cap  \dot H_{r^{-1} +r}^{1} (\RR^+) \cap \dot H^2(\RR^+). 
\label{l=2: approx sequence}
\ee
Recall $\tilde \calL_{l, \a}$ from \eqref{tildeLlapha} with $l = 2, \a = 0.2$, we claim that 
\be
(\tilde \calL_{l,\alpha} q_n,q_n)_{L^2(\RR^+)} \to (\tilde \calL_{l,\alpha} q,q)_{L^2(\RR^+)}, \quad \text{ as } n \to \infty. \label{eqconvtildecalLqn}
\ee
Then we can deduce by Proposition \ref{propL2coer} that
\[
\Re \left( \tilde\calL_{l,\alpha} q, q \right)_{L^2} = \lim_{n \to\infty} \Re \left( \tilde\calL_{l,\alpha} q_n, q_n \right)_{L^2}
\ge \lim_{n \to \infty} \epsilon_0 \| q_n \|_{L^2}^2  = \epsilon_0 \| q \|_{L^2}^2,
\]
which contradicts to our assumption if $\tilde{\delta} < \epsilon_0$.

Thus it remains to prove \eqref{eqconvtildecalLqn}. In fact, once we expend $\tilde \calL_{l, \a}$, we find that the convergence for local terms under \eqref{l=2: approx sequence} are obvious. It suffices to check the following bound for nonlocal term  $T_2$ (see \eqref{Tl: def} with $l=2$):
\be
\| r^{0.2} T_2 r^{-0.2} g \|_{L^2(\RR^+)} \lesssim \| g \|_{L^2(\RR^+)}.
\label{l=2: convergence of nonlocal term}
\ee 
In fact, by Cauchy Schwarz inequality, we find the pointwise estimate
\begin{align*}
    & \quad \Big| r^{0.2} D_4^{-1} \left(V_1 r^{-0.2} g \right) \Big| 
    = \Bigg| r^{-3.8} \int_0^r \frac{8 s^{4.8}}{(s^2 +2)^2}g(s)ds \Bigg| \\
    & = r^{-3.8} \left( \int_0^r  \frac{64 r^{9.6}}{(s^2 +2)^4} ds \right)^\frac{1}{2} \| g \|_{L^2} \lesssim \begin{cases}
         r^{1.5} \| g\|_{L^2}, & r \le 1, \\
         r^{-2.5} \|g \|_{L^2}, & r \ge 1,
    \end{cases}
\end{align*}
and similarly, 
\begin{align*}
     & \Big| V_2 D_{-2}^{-1} (r^{-0.2} g) \Big|
     \lesssim \frac{r^2}{\la r \ra^4}  \left( \int_r^\infty s^{-4.4} ds \right)^\frac{1}{2}  
 \| g \|_{L^2} \lesssim \frac{ r^{0.3}}{\la r \ra^4} \| g \|_{L^2},\\
    &  \Big| r^{0.2} D_4^{-1} \left( V_2 D_{-2}^{-1} \left(r^{-0.2} g\right) \right) \Big|
    \lesssim  r^{-3.8} \left( \int_0^r \frac{s^{4.3}}{\la s \ra^4} ds \right) \| g\|_{L^2} 
     \lesssim \begin{cases}
         r^{1.5} \| g\|_{L^2}, & r \le 1, \\
         r^{-2.5} \| g \|_{L^2},  & r \ge 1.
    \end{cases}
\end{align*}
The estimates above immediately yield \eqref{l=2: convergence of nonlocal term} and therefore conclude \eqref{eqconvtildecalLqn}.

\mbox{}

\noindent \underline{\textbf{Case 3. $l=1$.}} 

\mbox{}

\underline{\textit{Step $1$.}} We assume that there exists a nontrivial solution
\be
f \in \Big( H^k(\RR^3) \cap \calH_{(1)} \setminus \text{span}\{ \partial_{x_1} Q, \partial_{x_2} Q, \partial_{x_3} Q \} \Big) \bigcap \left( \bigcup_{n>0} \ker \left( \calL -z \right)^n \right),
\label{L1contraass}
\ee
with $\Re z < \tilde{\delta}$. Then we can find $n_0 \ge 1$ such that
\[
(\calL -z)^{n} f \not =0, \quad \forall \;   0 \le n <n_0, \quad \text{ and } \quad (\calL -z)^{n_0} f =0, 
\]
Similar to the discussion for eigenfunctions, by \cite[Lemma $2.9$]{ksnsblowup}, we have $f(x) = f_{1,m}(r) Y_{1,m}(\Omega) \in H^\infty(\RR^3) \cap \calH_{(1)}$. In addition, the eigen equation further implies $\calL^n f \in H^\infty(\RR^3)$ for any $n \in \ZZ_{\ge 0}$, and therefore $\calL^n f \in \calD_1$ from Proposition \ref{propwaveop2} for any $n \in \ZZ_{\ge 0}$.

Therefore, the function $q := T f_{1, m} \in H^\infty_{loc}(\RR^+)$ and the commutator formula \eqref{eqwaveopcom} indicates  
\be
\left( \tilde \calL_1 -z \right)^{n_0} q = \left( \tilde \calL_1 -z \right)^{n_0} Tf_{1,m} = T(\calL_1 -z)^{n_0} f_{1,m} =0, {\quad {\rm pointwisely \,\, on\,\,}\RR^+}, \label{eqeigeneqq}
\ee
for $n_0 \ge 1$ determined previously.

\underline{\textit{Step $2$.}} 
We claim that for any $G(x) =G_{1,m}(r) Y_{1,m}(\Omega) \in H^\infty(\RR^3) \cap \calH_{(1)}$, if we define $q_G := T G_{1,m}$, then
\be q_G\in  L_{1+ r^{-2}}^2(\RR^+) \cap \dot H^2(\RR^+). \label{eqregq}\ee
In fact, from the explicit formula of $Q$ \eqref{profile}, we can compute the formula of $T$ in \eqref{waveop2} as
\[ T[G_{1,m}] = G_{1,m} -  \frac{r^2+10}{r^4(r^2+2)}\int_0^r G_{1,m} s^3 ds.
\]

We first consider $r \in [0, 1]$, where we notice the following cancellation near $0$:
\begin{align*}
	T[r] =r - \frac{r^2+10}{r^4(r^2+2)} \int_0^r s^4 ds = r- \frac{1}{5} \cdot \frac{r(r^2+10)}{r^2+2} = -\frac{4r^3}{5(r^2+2)}.
\end{align*}
Then, applying the decomposition \eqref{flmdecayeq1} with $l = 1$ for $G_{1, m}$ with $r \in [0, 1]$, we control
\begin{align*}
	& |Th_{1,m}(r)| = \Big|  h_{1,m}(r) - \frac{r^2+10}{r^4(r^2+2)} \int_0^r h_{1,m}(s) s^3 ds \Big|  \lesssim  r^2+ \frac{1}{r^4} \int_0^r s^5 ds \lesssim r^2, \\
	& |\partial_r Th_{1,m}(r)|
	\lesssim |\partial_r h_{1,m}(r)| + \frac{1}{r^5} \int_0^r s^5 ds + \frac{1}{r} |h_{1,m}| \lesssim r, \\
	& |\partial_r^2 Th_{1,m}(r)| \lesssim |\partial_r^2 h_{1,m}(r)|
	+  \frac{1}{r^6} \int_0^r s^5 ds + \frac{1}{r^2} |h_{1,m}|+ \frac{1}{r} |\partial_r h_{1,m}| \lesssim 1,
\end{align*}
where $h_{1,m}$ is the residual term from \eqref{flmdecayeq1} and we have used \eqref{flmdecayeq2}. Hence for $j = 0, 1, 2$,\footnote{Here we exploited that $Tf(r) = (T f|_{[0,r]})(r)$, so that the decomposition \eqref{flmdecayeq1} on $[0,1]$ propagates to $Tf \big|_{[0, 1]}$.}
\be 
|\partial_r^jq_G(r) |
= |\partial_r^jT G_{1,m}(r)| = \Big| A_{1,m} \partial_r^j T [r] + \partial_r^j Th_{1,m} \Big| \lesssim r^{2-j}, \quad \text{ when } r\in (0,1).
\label{Tf1mregu}
\ee
When it is away from the origin with $r \ge 1$, we estimate
\begin{align*}
	& |q_G(r)| =  \Big|  G_{1,m}(r) - \frac{r^2+10}{r^4(r^2+2)} \int_0^r G_{1,m}(s) s^3 ds \Big|
	  \lesssim |G_{1,m}(r)| + \frac{1}{r^4} \| G_{1,m} \|_{L^2(\RR^+;r^2 dr)} \left( \int_0^r s^4 ds\right)^\frac{1}{2} \\
	& \; \qquad \lesssim   |G_{1,m}(r)| + r^{-1.5}  \| G_{1,m} \|_{L^2(\RR^+;r^2 dr)} \lesssim |G_{1,m}(r)|  + r^{-1.5}, \\
	& |\partial_r q_G(r)| \lesssim |\partial_r G_{1,m}(r)| + \frac{1}{r^5}
	\int_0^r |G_{1,m}(s)| s^3 ds + \frac{|G_{1,m}|}{r} \lesssim 
	|\partial_r G_{1,m}| + \frac{|G_{1,m}|}{r} + r^{-2.5}, \\
	& |\partial_r^2 q_G(r)| \lesssim |\partial_r^2 G_{1,m}(r)| + \frac{|\partial_r G_{1,m}|}{r} + \frac{|G_{1,m}|}{r^2}  + \frac{1}{r^6}  \int_0^r |G_{1,m}(s)| s^3 ds \\
	& \; \quad \qquad 
	\lesssim |\partial_r^2 G_{1,m}(r)| + \frac{|\partial_r G_{1,m}|}{r} + \frac{|G_{1,m}|}{r^2} + r^{-3.5}.
\end{align*}
Consequently, combining all of the estimates above and Lemma \ref{lemflmL2integrable} deduces \eqref{eqregq}.

In particular, with the claim constructed above, if we apply this result with $G = \calL^n f = \calL_1^n f_{1,m}(r) Y_{1,m} (\Omega) \in H^\infty(\RR^3) \cap \calH_{(1)}$ for any $n \ge 0$, then we find that
\be
T [\calL_1^n f_{1,m}] = T [(\calL^n f)_{1,m}] \in L_{1+ r^{-2}}^2(\RR^+) \cap \dot H^2(\RR^+), \quad \forall \;  n \ge 0.
\label{l=1: gener eigen}
\ee

\underline{\textit{Step $3$.}} As for $U_1(r)= \frac{e^{\frac{1}{8}r^2}}{r(2+r^2)}$ defined in \eqref{U1def}, we claim that the function $0 \not= p :=U_1^{-1} q$ satisfies 
\[
\tilde \calL_1'^n p \in D(\calL_1') \text{ for any } n \ge 0  \text{ and } \left( \tilde \calL_1' -z \right)^{n_0} p =0,
\]
with $D(\calL_1')$ defined in \eqref{domainofclosure}. In fact, from the commutator relations \eqref{eqwaveopcom} and \eqref{conjugationtildeL1} together with \eqref{l=1: gener eigen} and that $\la r \ra^2 \pa_r^s U_1^{-1} \in L^\infty$ for any $s \ge 0$, we find that
\begin{align*}
    \tilde \calL_1'^n p & = \tilde\calL_1'^n U_1^{-1} Tf_{1,m} 
= U_1^{-1} \tilde \calL_1^n Tf_{1,m} 
= U_1^{-1} T \calL_{1}^n f_{1,m} \\
&= U_1^{-1} T[\calL_1^n f_{1,m}] \in L_{r^2 + r^{-2}}^2(\RR^+) \cap \dot H^2(\RR^+), \quad \forall \;  n \ge 0.
\end{align*}
which, yields that $\tilde \calL_1'^n p \in D(\tilde \calL_1')$ for any $n \ge 0$. 

In addition, together with \eqref{conjugationtildeL1} and \eqref{eqeigeneqq},
\[
\left( \tilde \calL_1' -z \right)^{n_0} p = 
\left( \tilde \calL_1' -z \right)^{n_0} U_1^{-1} q 
= U_1^{-1} (\tilde \calL_1 -z)^{n_0} q =0, \quad \text{ pointwisely on $\RR^+$},
\]
Hence, combining the above arguments yields the claim.

\underline{\textit{Step $4$.}} With the claim given in \textit{Step 3}, we find a generalized eigenfunction of $\tilde \calL_1' : D(\tilde \calL_1' ) \subset L^2(\RR^+) \to L^2(\RR^+)$ with respect to eigenvalue $z$, which means that $z \in \sigma(\tilde \calL_1' ) \subset \Big[\frac{2}{5},\infty\Big)$ (see \eqref{spectrumL1tt}). This leads to a contradiction if $\tilde {\delta} < \frac{2}{5}$.

\mbox{}

\noindent \underline{\textbf{Case 4. $l = 0$.}} This was verified in \cite[Theorem 4.1]{MR4685953}. 

\mbox{}

In summary, we have concluded the proof of Theorem \ref{thm: mode stability main} for some $0 < \tilde{\delta} \ll 1$.
\end{proof}

\subsection{Linear theory of $-\calL$}\label{sec26} As the end of Section \ref{sectionmodesta}, we summarize the spectral property of the linearized operator $-\calL$ as follows, which is deduced by the combination of Theorem \ref{thm: mode stability main} with the abstract result \cite[Proposition $2.7$, Proposition $2.8$, Corollary $2.13$]{ksnsblowup}.

\begin{proposition}[Linear theory of $-\calL$]
\label{proposition: spectral properties of L}
    For $k \ge 0$ and $\calL: D\left(\calL\Big|_{H^k}\right) \subset H^k(\RR^3;\CC) \to H^k(\RR^3;\CC)$ given in \eqref{eqdefcalL}, there exists $\delta_g \in \left(0, \min\{ \tilde{\delta}, \frac{1}{16} \} \right)$, such that followings hold:

    \noindent $(1)$ Unstable eigenmodes. The set $-\Lambda_\calL := \sigma(-\calL) \cap \{ z\in \CC: \Re z > - \frac{\delg}{2} \} \subset \sigma_{disc}(-\calL)$. Precisely,
\be
   -\Lambda_\calL = \Big\{ 1, \frac{1}{2} \Big\}.
   \label{spectralofL}
\ee
Moreover, the real-valued functions $\{ \varphi_j \}_{j=0}^3 \subset H^\infty(\RR^3)$ defined by
\be
\varphi_0 = \frac{\Lambda Q}{\| \Lambda Q \|_{H^k}} \text{ and }
\varphi_j = \frac{\partial_{x_j} Q}{\| \partial_{x_j} Q \|_{H^k}}, \text{ with } 1\le j \le 3,
\label{baseofunstable}
\ee
form the basis of the unstable eigensubspace as \eqref{equnstabspec2}.
% , i.e.
% \be
% V = \bigoplus_{z \in -\Lambda_\calL} \ker \left( -\calL - \lambda \right)^{\mu_z} = \text{span}_{\CC} \{ \varphi_j \}_{j=0}^3, \quad \varphi_0 \in \ker \left(-\calL -1 \right) \text{ and } \varphi_1,\varphi_2,\varphi_3 \in \ker \left(-\calL -\frac{1}{2} \right),
% \label{gereigensp}
% \ee
% where $\mu_z>0$ is the smallest integer such that $\ker (-\calL -z  )^{\mu_z} = \ker (-\calL -z)^{\mu_z +1}$.

\mbox{}

\noindent $(2)$ Riesz projection $P_u$ onto unstable eigenspace. Define the spectral projection of $-\calL$ to the set $-\Lambda_\calL$:
\be
P_u = \frac{1}{2\pi i}\int_\Gamma (z - (-\calL))^{-1} dz, \quad \text{ and } P_s = \text{Id} - P_u,
\label{Pu: definition}
\ee
where $\Gamma \subset \rho(-\calL)$ is an arbitrary anti-clockwise contour containing $-\Lambda_\calL$, then $V = \text{Ran}\; P_u$. In addition, there exists $\{\psi_j\}_{j=0}^3 \subset H^k(\RR^3;\RR)$ such that
\be
(\varphi_i, \psi_j)_{H^k} = \delta_{ij}, \quad  \text{ and } \quad  P_u f = \sum_{i=0}^3 (f,\psi_j)_{H^k} \varphi_j, \quad \forall  f\in D(\calL).
\label{Puinnerproduct}
\ee
Consequently, for any $k' \ge k$, $P_u: H^k \to H^{k'}$, $P_s: H^{k'} \to H^{k'}$ are bounded. 

\mbox{}

\noindent $(3)$ Direct decomposition of $H^k(\RR^3;\RR)$. The spaces
\be
H_u^k = \text{Ran} \left( P_u\Big|_{H^k(\RR^3;\RR)} \right), \quad \text{ and } \quad H_s^k :=\text{Ran} \left( P_s \Big|_{H^k(\RR^3;\RR)} \right),
\label{defHuHs}
\ee
 are both invariant under $-\calL$. In addition, $H^k(\RR^3;\RR)$ follows the direct sum decomposition  
 \[H^k(\RR^3;\RR) = H_u^k \oplus H_s^k.
 \]

\mbox{}

 \noindent $(4)$ Growth rate of unstable part. There exists an inner product on $H_u^k$, which we denote by $(\cdot, \cdot)_{\tilde B}$ such that
\be
(-\calL f,f)_{\tilde B} \ge - \frac{6 \delta_g}{10} \| f\|_{\tilde B}^2, \quad \forall f \in H_u^k.
\label{normunstable}
\ee
Moreover, $(\cdot , \cdot)_{\tilde B}$ induces a norm $\| \cdot \|_{\tilde B}$ on $H_u^k$ and $\| \cdot \|_{\tilde B} \sim \| \cdot \|_{H^n}$ for any $n \ge 0$.

\mbox{}

\noindent $(5)$ Semigroup decay of stable part. $-\calL$ generates a semigroup, which we denote by $e^{-\tau \calL}$, and it satisfies
\be
\| e^{-\tau \calL} v\|_{H^k} \lesssim e^{-\frac{1}{2} \delg \tau} \| v \|_{H^k}, \quad \forall \; v \in H_s^k,\quad \tau \ge 0.
\label{decay rate: stable part}
\ee
    
\end{proposition}

\section{Nonlinear stability}
\label{sectionnonsta}

In this section, we come back to the nonlinear equation \eqref{linearizedeq} in self-similar coordinates and prove Theorem \ref{thm: stability of ss solu}, i.e. the non-radial nonlinear stability of $Q$ in $H^2(\RR^3)$.

\subsection{Finite codimensional stability.}
\label{subsfinitecodista}

Ensured by Proposition \ref{proposition: spectral properties of L} with $k = 1$, if we project the equation \eqref{linearizedeq} onto stable and unstable subspaces respectively via the Reisz projections $P_s$ and $P_u$ defined in \eqref{Pu: definition}, then the variables $(\eps, \epu) = (P_s \ep, P_u \ep)$ solve the system
\be
\begin{cases}
    \partial_\tau \eps = -\calL \eps + P_s N(\ep), \\
    \partial_\tau \epu = -\calL \epu + P_u N(\ep).
\end{cases}
\label{linearizedeqV2}
\ee
where $\calL$, $N(\ep)$ are defined in \eqref{eqdefcalL}, \eqref{nonlinearterm} respectively.

\begin{proposition}[Bootstrap]
As for $\delta_g>0, P_s, P_u, H_s^1, H_u^1$ determined by Proposition \ref{proposition: spectral properties of L} with $k = 1$, there exist $0 \le \delta_i \ll 1 \; (0 \le i \le 3)$ with
\be
\delta_0 \ll \delta_3 \ll \delta_1 \ll \delta_2 \ll \delta_3^\frac{1}{2} \ll \delta_g \ll 1,
\label{bootstrapcoeffi}
\ee
\label{propbootstrap}
such that for any initial datum $\ep_{s}(0) =\ep_{s0} \in H_s^1 \cap H^2$ satisfying
\be
\| \ep_{s0} \|_{H^2} \le \delta_0, 
\label{bootstrapinitial}
\ee
there exists $\ep_{u0} \in H_u^1$ (we denote it by $\Phi(\ep_{s0}):= \ep_{u0}$ in later discussion) such that the solution to \eqref{eqrenormalized} globally exists and satisfies the followings for all $\tau \ge 0$:

\noindent
$\bullet$ (Control of the stable part $\eps$)
\be
\| \eps \|_{H^1} < \delta_1 e^{-\frac{1}{2} \dg \tau},
\label{bootstrapstaH3}
\ee

\noindent
 $\bullet$  (Control of the higher regularity of the stable part $\eps$)
\be
\| \eps \|_{\dot H^2} < \delta_2 e^{-\frac{1}{2} \dg \tau},
\label{bootstrapstaH5}
\ee

\noindent
 $\bullet$  (Control of the unstable part $\epu$)
\be
\| \epu \|_{\tilde B} < \delta_3 e^{-\frac{7}{10} \dg \tau}.
\label{bootstrapuns}
\ee
\end{proposition}

The proof of Proposition \ref{propbootstrap} is almost the same as \cite[Section $3$]{ksnsblowup}, with further simplification for neither being coupled with the fluid equation nor applying modified spectral projection. 
We include the proof out of completeness. 

We will apply the boundedness of $P_u$, $P_s$ on $H^1$ or $H^2$ from Proposition \ref{proposition: spectral properties of L}(2) without citing. We also remark that the constant $C$ appearing below might vary from line to line, but is always independent of $\delta_i$ $(0\le i \le 3)$.

\subsubsection{$H^1$ estimate of $\eps$ via semigroup method.}
\begin{lemma}
\label{lem: nonlinear norm Hk}
	For any $f \in H^{k+1}(\RR^3)$ with $k \ge 1$, then $N(f) \in H^{k}(\RR^3)$, in addition,
	\be
	\| N(f) \|_{H^k} \lesssim \| f \|_{H^{k+1}}^2.
	\label{nltermesti} 
	\ee
\end{lemma}
\begin{proof}
Since $f \in H^{k+1}(\RR^3)$ with $k \ge 1$, by Sobolev inequality, $f \in L^{\infty}(\RR^3)$ and $\| f\|_{L^{\infty}} \lesssim \| f \|_{H^{k+1}}$. Hence 
\begin{align*}
\|\na \Delta^{-1} f \|_{L^\infty}
& =  \sup_{x \in \RR^3}\Big| -\frac{1}{4 \pi} \na_x \int \frac{f(y)}{|x-y|} dy \Big| \lesssim \sup_{x \in \RR^3}\int \frac{|f(x+y)|}{|y|^2} dy  \\
& \lesssim  \left( \int_{|y| \le 1} \frac{1}{|y|^2} dy \right) \| f \|_{L^\infty} + \left( \int_{|y| \ge 1} \frac{1}{|y|^4} dy \right)^\frac{1}{2} \| f \|_{L^2} \lesssim \| f\|_{H^{k+1}},
\end{align*}
where the coefficient is uniform in $f \in H^{k+1}$. Consequently,
	\begin{align*}
		\| N(f) \|_{H^k}
		& = \| f^2 + \na f \cdot \na \Delta^{-1} f \|_{H^{k+1}} 
		 \le \| f^2 \|_{H^{k+1}} + \| \na f \cdot \na \Delta^{-1} f \|_{H^k} \\
		 & \lesssim \| f \|_{H^{k+1}}^2 
         + \|  f \|_{H^{k+1}}\|  \na \Delta^{-1} f \|_{L^\infty}
         %+ \| \na f\|_{L^\infty} \| \na \Dein f \|_{\dot H^1 \cap \dot H^k}
          +
         \sum_{|\alpha| + |\beta| =k, \;  |\alpha| < k}  \| \partial^\alpha \na f \|_{L^4} \| \partial^\beta \na \Dein f \|_{L^4} \\
         & \lesssim \|f\|_{H^{k+1}}^2.
	\end{align*}
\end{proof}

\begin{lemma}[Semigroup estimate of $\| \eps \|_{H^1}$]
\label{lemimproveH3}
Under the assumptions of Proposition \ref{propbootstrap}, the bootstrap assumption \eqref{bootstrapstaH3} can be improved to
\be
\| \eps(\tau) \|_{H^1} \le \frac{1}{2} \delta_1 e^{-\frac{\delta_g}{2} \tau}, \quad \forall \; \tau \in [0,\tau^*].
\label{improveH3}
\ee
\end{lemma}
\begin{proof}
    From the first equation of \eqref{linearizedeqV2}, by using Duhamel's principle, $\eps$ can be written as
 \[
 \ep_s(\tau) = e^{-\tau \calL } \ep_{s0} + \int_0^\tau e^{-(\tau-s) \calL}P_s(N(\ep) )(s) ds.
  \]
  We now apply \eqref{decay rate: stable part}, \eqref{nltermesti},  \eqref{bootstrapcoeffi}, \eqref{bootstrapstaH3},  \eqref{bootstrapstaH5} and \eqref{bootstrapuns} to see
  \begin{align*}
  	\| \ep_s \|_{H^1}
  	& \le C e^{-\frac{\dg}{2} \tau} \| \ep_{s0} \|_{H^1} + C\int_0^\tau e^{-\frac{\dg}{2} (\tau-s)} \left( \| \ep_s \|_{H^2}^2 +  \| \ep_u \|_{H^2}^2  \right)ds \\
  	& \le C\left( \delta_0 + \delta_1^2 + \delta_2^2 + \delta_3^2 \right) e^{-\frac{\dg}{2} \tau} \le \frac{1}{2} \delta_1 e^{-\frac{\dg}{2} \tau}, \quad \forall \; \tau  \in [0,\tau^*].
  \end{align*}
  \end{proof}

  \subsubsection{$\dot H^2$ estimate of $\eps$ via energy estimate.}

  \begin{lemma}[A priori estiamte of  $\| \eps \|_{\dot H^2}$]
\label{lemimproveH5}
Under the assumptions of Proposition \ref{propbootstrap}, the bootstrap assumption \eqref{bootstrapstaH5} can be improved to
\be
\| \eps(\tau) \|_{\dot H^2} \le \frac{1}{2} \delta_2 e^{-\frac{\delta_g}{2} \tau}, \quad \forall \; \tau \in [0,\tau^*].
\label{improveH4}
\ee
\end{lemma}

  \begin{proof}
  Since Keller-Segel equation \eqref{equation, Keller-Segel} is locally well-posed in $H^2$\footnote{This can be obtained from standard fixed point arguments as in \cite{zbMATH01532872} or \cite[Appendix A]{ksnsblowup}.}, so is the renormalized system \eqref{eqrenormalized} and the regularity of $\ep(\tau, \cdot)$ can be improved to $H^3$ for any $\tau>0$ by the smoothing effect of heat kernel $e^{t \Delta}$. Then by integration by parts and interpolation inequality, we obtain that
  \begin{align*}
      & \left( (\Delta - \frac{1}{2} \Lambda) \ep_s,  \ep_s \right)_{\dot H^2}
      = - \| \ep_s \|_{\dot H^3}^2  - \frac{5}{4} \| \ep_s \|_{\dot H^2}^2, \\
      & \left(\na \cdot (\ep_s \na \Dein Q), \ep_s \right)_{\dot H^2}
      = \frac{1}{2} \int Q |\calD^2 \ep_s|^2 dx + O\left( \| \ep_s \|_{H^1} \| \ep_s \|_{\dot H^3} \right) = O \left( \| \ep_s \|_{H^1} \| \ep_s \|_{\dot H^3} \right), \\
      & \left(\na \cdot (Q \na \Dein \ep_s), \ep_s \right)_{\dot H^2}
      = O \left( \| \ep_s \|_{H^1} \| \ep_s \|_{\dot H^3} \right),
  \end{align*}
  which, by Young's inequality, yields that
      \begin{align}
	-(\calL \eps , \eps)_{\dot H^2}
	& \le -\frac{7}{8} \| \eps \|_{\dot H^3}^2 - \frac{5}{4} \| \eps \|_{\dot H^2}^2 + C \| \eps \|_{H^1}^2,
    \label{H4liesti}
\end{align}
for some $C>0$. We also estimate using \eqref{nltermesti} to see 
\begin{align*} \left( P_s N(\ep) , \ep_s \right)_{\dot H^2} \lesssim \| N(\e) \|_{H^2} \| \e_s \|_{\dot H^2} \lesssim \left( \| \e_s \|_{\dot H^3}^2  + \| \e_s \|_{H^2}^2  + \| \e_u \|_{\tilde B}^2 \right) \| \e_s \|_{\dot H^2}.
\end{align*}
Thus taking $\dot H^2(\RR^3)$ inner product with $\eps$ onto both sides of first equation of \eqref{linearizedeqV2}, and by \eqref{bootstrapstaH3},  \eqref{bootstrapstaH5}, \eqref{bootstrapuns},
    \begin{align*}
	\frac{1}{2} \frac{d}{d\tau} \| \eps \|_{\dot H^2}^2
	& = \left( -\calL\eps + P_s N(\ep) , \ep_s \right)_{\dot H^2} \\
    &  \le \left(C \| \e_s\|_{\dot H^2}  -\frac{7}{8}\right) \| \eps \|_{\dot H^3}^2 - \frac{5}{4} \| \eps \|_{\dot H^2}^2 + C \| \eps \|_{H^1}^2  \\
    & \qquad + 
	 C \| \ep_s \|_{H^2}^2 \| \eps \|_{\dot H^2} +  C \| \epu \|_{\tilde B}^2 \| \eps \|_{\dot H^2} \\
    & \le - \| \eps \|_{\dot H^2}^2 + C \delta_1^2 e^{-\delta_g \tau} + C(\delta_1 + \delta_2+ \delta_3)^2 \delta_2 e^{-\frac{3\delta_g}{2} \tau}, \quad \forall \; \tau \in (0,\tau^*],
 \end{align*}
 and thus by Gronwall's inequality, one deduces that
 \begin{align}
     \| \eps(\tau) \|_{\dot H^2}^2  
     & \le e^{-2 (\tau- \tau_0)}\| \ep_{s} (\tau_0) \|_{\dot H^2}^2 +  \int_{\tau_0}^\tau e^{-2 (\tau-s)} \left(  C \delta_1^2 e^{-\delta_g s} + C(\delta_1 + \delta_2+ \delta_3)^2 \delta_2 e^{-\frac{3\delta_g}{2} s} \right) ds \nonumber \\
     & \le e^{-\delta_g (\tau-\tau_0)} \| \ep_{s} (\tau_0) \|_{\dot H^2}^2 + 
 C \left( \delta_1^2 + (\delta_1 + \delta_2+ \delta_3)^2 \delta_2 \right) e^{-\delta_g (\tau - \tau_0)}, \quad \forall \; \tau \in [\tau_0,\tau], \label{eqenergyintegral}
 \end{align}
 with any $\tau_0 >0$. By the local wellposedness of \eqref{equation, Keller-Segel} again, we know that $\rho \in C_t^0 H_x^2\left( \left[ 
 0,T_{max} \right) \times \RR^3 \right)$, which propagates \eqref{eqenergyintegral} to $\tau_0 = 0$.  Consequently,
  \begin{align*}
     \| \eps(\tau) \|_{\dot H^2}^2 
     & \le e^{-\delta_g \tau} \| \ep_{s0} \|_{\dot H^2}^2 + C \left( \delta_1^2 + (\delta_1 + \delta_2+ \delta_3)^2 \delta_2 \right) e^{-\delta_g \tau} \le \frac{\delta_2^2}{4} e^{-\delta_g \tau}, \quad \forall \; \tau \in [0,\tau],
 \end{align*}
 under the assumption of the coefficients $\{ \delta_i \}$ given in \eqref{bootstrapinitial}, and thus we have completed the proof.
\end{proof}

\subsubsection{Control of unstable part $\epu$.} 
\begin{proof}[Proof of Proposition \ref{propbootstrap}]

\mbox{}

\noindent \underline{\textit{Improvement of the bootstrap assumptions.}} 
By contradiction, we assume that there exists $\ep_{s0}$ satisfying (\ref{bootstrapinitial}), such that for any $\ep_{u0}$ with $\| \ep_{u0} \|_{\tilde B} \le \delta_3$, the exit time $\tau^*$ defined by	
\be
\tau^* = \sup \{ \tau \ge 0:  (\eps, \epu) \text{ satisfy \eqref{bootstrapstaH3}-\eqref{bootstrapuns} simultaneously on } [0,\tau] \},
\label{exiting time}
\ee
is finite. Then according to Lemma \ref{lemimproveH3} and Lemma \ref{lemimproveH5}, combining with the continuity of the solution to system \eqref{linearizedeqV2} with respect to $\tau$, there exists $0 < \tau_\epsilon \ll 1$ such that (\ref{bootstrapstaH3}) and (\ref{bootstrapstaH5}) hold for all $\tau \in [0, \tau^* + \tau_\epsilon]$. 

\mbox{}

\noindent \underline{\textit{Out-going flux property.}}
  With the argument above, we see that the only scenario for the solution to exit the bootstrap regime is when the bootstrap assumption for the unstable part $\epu$ given in \eqref{bootstrapuns} does not hold for $\tau > \tau^*$. Hence if we define 
\[
\tilde B(\tau) = \{ v \in H_u^1: \| v \|_{\tilde B} \le \delta_3 e^{-\frac{7}{10} \dg \tau} \},
\]
then we find that 
\[
\begin{cases}
	\epu(\tau) \in \tilde B(\tau), & \forall \; \tau \in [0,\tau^*], \\
	\epu(\tau^*) \in \partial \tilde B(\tau^*), & \tau =\tau^*.
\end{cases}
\]
Taking $\tilde B$ inner project with $\epu$ onto both sides of the second equation of \eqref{linearizedeqV2}, together with \eqref{normunstable} and \eqref{nltermesti},
\begin{align*}
	\frac{1}{2} \frac{d}{d\tau} \| \epu \|_{\tilde B}^2 
	& = \left( \calL \epu + P_uN(\ep) , \epu \right)_{\tilde B} \\
	& \ge - \frac{6 \dg}{10} \| \epu \|_{\tilde B}^2 + \left(P_uN(\ep) , \epu \right)_{\tilde B} \\
	& \ge - \frac{6 \dg}{10} \| \epu \|_{\tilde B}^2  - \left( \| \eps \|_{H^2}^2 + \| \epu \|_{\tilde B}^2 \right) \| \epu \|_{\tilde B}.
\end{align*}
Consequently, by \eqref{bootstrapcoeffi}, \eqref{bootstrapinitial}, \eqref{bootstrapstaH3}, \eqref{bootstrapstaH5} and \eqref{bootstrapuns},  
\begin{align}
	\frac{d}{d\tau}\Big|_{\tau = \tau^*} \left( e^{\frac{7}{5} \dg \tau} \| \epu \|_{\tilde B}^2 \right)
	& = \left( \frac{7 \dg}{5} e^{\frac{7}{5} \dg \tau} \| \epu \|_{\tilde B}^2 + e^{\frac{7}{5} \dg \tau}  \frac{d}{d\tau}\| \epu \|_{\tilde B}^2 \right) \Big|_{\tau =\tau^*} \notag \\
	& \ge \left( \frac{\dg }{5}  e^{\frac{7}{5} \dg \tau}  \| \epu \|_{\tilde B}^2  - C e^{\frac{7 \dg}{5} \tau}\left( \delta_1^2 + \delta_2^2 + \delta_3^2 \right) \delta_3 e^{-\frac{3}{2} \delg \tau} \right) \Big|_{\tau = \tau^*} \notag \\
	&  \ge \frac{\dg}{5} \delta_3^2 - C \delta_3 \left( \delta_1^2 + \delta_2^2 +\delta_3^2 \right) \ge \frac{\dg}{10} \delta_3^2 >0,
	\label{outgoing flux}
\end{align}
which is called the outer-going flux property.

\mbox{}

\noindent  \underline{\textit{Brouwer's topological argument.}}
By the $H^2$ local well-posedness theory of $3D$ Keller-Segel equation and the out-going flux property \eqref{outgoing flux}, we have the continuity of the mapping $\e_{u0} \mapsto \epu \left(\tau^*(\ep_{u0}) \right)$, where $\tau^*(\ep_{u0})$ denotes the exit time with initial data $\ep_{u0}$. Subsequently, we define a continuous map $\Phi:\overline{B_{\tilde B}(0,1)} \to \partial {B_{\tilde B}(0,1)}$ as follows:
\[
f \in \overline{B_{\tilde B}(0,1)}  \mapsto \ep_{u0}:= \delta_3 f \in \overline{B_{\tilde B}(0, \delta_3)} \mapsto \epu(\tau^*) \in \partial \tilde \calB(\tau^*) \mapsto \frac{\epu(\tau^*)}{\| \epu(\tau^*) \|_{\tilde B}} \in \partial B_{\tilde B}(0,1),
\]
where $B_{\tilde B}(0,1)$ represents the unit ball centered at the origin in the finite-dimensional space $(H_u^1, \| \cdot \|_{\tilde B})$. Particularly, when $f \in \partial B_{\tilde B}(0,1)$, i.e., $\delta_3 f \in \partial B_{\tilde B} (0,\delta_3)$, according to the outer-going flux property of the flow as indicated in (\ref{outgoing flux}), $\tau =0$ is the exit time, thus leading to $\Phi(f) = f$. In essence, $\Phi= Id$ on the boundary $\partial B_{\tilde B}(0,1)$. However, upon applying Brouwer's fixed-point theorem to $-\Phi$ on $B_{\tilde B}(0,1)$, we conclude the existence of a fixed point of $-\Phi$ on the boundary $\partial B_{\tilde B}(0,1)$, contradicting the assertion that $\Phi =Id$ on the boundary. Thus we have concluded the proof of Proposition \ref{propbootstrap}.
\end{proof}

\subsection{Lipschitz dependence}
\label{subsLipschizde}
\begin{proposition}[Lipschitz dependence]
\label{propLipde} 
    Let $\ep_{s0}^{(0)},\ep_{s0}^{(1)} \in H_s^1(\RR^3) \cap H^2(\RR^3)$ satisfy \eqref{bootstrapinitial}, then the related initial datum of the unstable part $\Phi(\ep_{s0}^{(0)}) = \ep_{u0}^{(0)}$ and $\Phi(\ep_{s0}^{(1)}) =\ep_{u0}^{(1)}$, associated by Proposition \ref{propbootstrap}, satisfy
    \be
    \Big\| \Phi(\ep_{s0}^{(0)}) - \Phi(\ep_{s0}^{(1)}) \Big\|_{\tilde B} \lesssim \bar \delta \Big\| \ep_{s0}^{(0)} - \ep_{s0}^{(1)} \Big\|_{H^2},
    \label{Lipsconti}
    \ee
    where $0< \bar\delta =\delta_1+ \delta_2 + \delta_3 \ll 1$ with $\delta_i$ provided in Proposition \ref{propbootstrap}.
\end{proposition}

% Before we prove the Lipschitz dependence, we introduce the following lemma which greatly helps simplify the later argument.
% \begin{lemma}
% \label{lemequivofnorm}
%     For every $k \ge 1$, there exists $C = C_k > 0$ such that
% \be
% \| f \|_{L^2} \le \| f \|_{H^k} \le C \| f \|_{L^2}, \quad  \forall f \in B_1^{H^k}(0).
% \label{equivL2andHk}
% \ee
% \end{lemma}
% \begin{proof}
% We define two normed vector spaces by
% \[
% G= \left( \overline{B_1^{H^k}(0)}, \| \cdot \|_{H^k} \right) \text{ and } E=\left( \overline{B_1^{H^k}(0)}, \| \cdot \|_{L^2} \right).
% \]
% It is easy to see that $G$ is a Banach space. Next, we claim that $E$ is a closed subspace of $L^2(\RR^3)$ and thus is also completed. In fact, assume there exists a sequence $\{ f_n \} \subset G$ such that $f_n \to f$ in $L^2$ sense for some $f \in L^2$. Since $\| f_n \|_{H^1} \le 1$ uniformly in $H^1$, there exists a subsequence (we still denote by $f_n$) and $\tilde f \in H^1$ such that $f_n \rightharpoonup \tilde f$ weakly in $H^1$ with $\| \tilde f \|_{H^1} \le \liminf_{n \to \infty} \| f_n \|_{H^1} \le 1$. By the uniqueness of the limit, $f = \tilde f \in E$.

% Then we construct an identical map $\text{Id}: G \to E$, which is bounded and bijective. Hence by inverse mapping theorem, we know that $\text{Id}^{-1}: E \to G$ is a bounded operator, which verifies \eqref{equivL2andHk}.

% \end{proof}

\begin{proof}[Proof of Proposition \ref{propLipde}]
    Recalling Proposition \ref{propbootstrap}, for any $\ep_{s0}^{(0)}$ and $\ep_{s0}^{(1)}$ satisfying \eqref{bootstrapinitial}, we can always find $\ep_{u0}^{(0)}$ and $\ep_{u0}^{(1)}$ such that the solutions to the system \eqref{linearizedeqV2} with initial datum $\Psi_0^{(i)} = Q + \ep_{s0}^{(i)} +\ep_{u0}^{(i)}$ $(i=1,2)$ satisfy the properties \eqref{bootstrapstaH3}, \eqref{bootstrapstaH5} and \eqref{bootstrapuns}. And for each $i \in \{ 1 , 2 \}$, $\ep^{(i)} = \eps^{(i)} + \epu^{(i)}$ solves the equation \eqref{linearizedeq}. As for the difference between $\ep^{(0)}$ and $\ep^{(1)}$, which is denoted by $\ep^D = \ep^{(0)} - \ep^{(1)}$, satisfies the equation
\be
	\partial_\tau \ep^D = \calL \ep^D + \na \cdot \left( \ep^D \na \Delta^{-1} \ep^{(0)} \right) - \na \cdot \left( \ep^{(1)} \na \Delta^{-1} \ep^D \right),
	\label{eqepdif}
\ee
where, out of simplicity, we denote the nonlinear term by
\be
N_D(\ep^D) = \na \cdot \left( \ep^D \na \Delta^{-1} \ep^{(0)} \right) - \na \cdot \left( \ep^{(1)} \na \Delta^{-1} \ep^D \right).
\label{difference: nonlinear term}
\ee
Next, we introduce
\be
\calA := \sup_{\tau \ge 0} e^{\frac{7}{10} \dg \tau} \big\| \ep_u^D \big\|_{\tilde B}, \quad \calE_1 := \sup_{\tau \ge 0} e^{\frac{\dg}{2} \tau} \big\| \ep_s^D \big\|_{H^1}, \quad
\calE_2 := \sup_{\tau \ge 0} e^{\frac{\dg}{2} \tau} \big\| \ep_s^D \big\|_{\dot H^2}.
% \sim \sup_{\tau \ge 0} e^{\frac{\dg}{2} \tau} \big\| \ep_s^D \big\|_{H^k}, 
\label{eqdefcalAE}
\ee
% where the last equivalence relationship holds for any $k \in [0,4]$ due to Lemma \ref{lemequivofnorm} with the boundedness from Proposition \ref{propbootstrap}.

\noindent 
\underline{\textit{Estimate of $\ep_s^D:= P_s \ep^D$.}} If we project the equation \eqref{eqepdif} onto $H_s^1(\RR^3)$, then we find that $\ep_s^D$ solves the equation
\be
\partial_\tau \ep_s^D = \calL \ep_s^D + P_s N_D(\ep^D),
\label{eqepdifsta}
\ee
which implies that
\[
\ep_s^D(\tau) = e^{\tau \calL } \ep_{s0}^D + \int_0^\tau e^{(\tau -s ) \calL } P_s N_D(\ep^D)(s) ds.
\]
By semigroup estimate \eqref{decay rate: stable part}, it yields that
\begin{align*}
	\big\| \ep_s^D(\tau) \big\|_{H^1}
	\le C e^{-\frac{\dg}{2} \tau} \big\| \ep_{s0}^D \big\|_{H^1} + C \int_0^\tau e^{-\frac{\dg}{2}(\tau-s)} \big\| N_D (\ep^D) \big\|_{H^1}(s) ds,
\end{align*}
and the nonlinear term, similar to the proof of \eqref{nltermesti}, can be estimated by
\begin{align}
	 \| N_D (\ep^D) \|_{H^1}
     \lesssim \| \ep^D \|_{H^2} \left( \| \ep^{(0)} \|_{H^2} + \| \ep^{(1)} \|_{H^2}\right).
     \label{Lipnonlinearesti}
\end{align}
Thus,
\begin{align}
   \big\| \ep_s^D(\tau) \big\|_{H^1}	
   & \lesssim e^{-\frac{\dg}{2} \tau} \big\| \ep_{s0}^D \big\|_{H^1} + \int_0^\tau e^{-\frac{\dg}{2} (\tau-s)} \big\| \ep^D \big\|_{H^2} \left( \big\| \ep^{(0)} \big\|_{H^2} + \big\| \ep^{(1)} \big\|_{H^2}\right) ds \notag\\
   & \lesssim e^{-\frac{\dg}{2} \tau} \big\| \ep_{s0}^D \big\|_{H^1} + \left( \delta_1+ \delta_2 + \delta_3 \right)  \int_0^\tau e^{-\frac{\dg}{2} (\tau-s)} \left( \calA + \calE_1 + \calE_2 \right) e^{-\dg s} ds \notag\\
   & \lesssim e^{-\frac{\dg}{2} \tau} \big\| \ep_{s0}^D \big\|_{H^1} + \bar\delta (\calA + \calE_1 + \calE_2) e^{-\frac{\dg}{2} \tau},\qquad \forall \,\, \tau \ge 0,
   \label{diffstaE1}
\end{align}
where the coefficients of the estimates are all uniform in $\ep^{(i)}$ and $\delta_j$.

\noindent
\underline{\textit{$\dot H^2(\RR^3)$ estimate of $\epu^D:= P_u \ep^D$.}} Recall \eqref{H4liesti}, we know that there exists $C_0>0$ such that
\[
(\calL \ep_{s}^D, \ep_s^D)_{\dot H^2(\RR^3)}\le -\frac{7}{8} \| \eps^D \|_{\dot H^3}^2 - \frac{5}{4} \| \eps^D \|_{\dot H^2}^2 + C_0^2 \| \eps^D\|_{H^1}^2.
\]

When it comes to the nonlinear estimate, by the same argument as Lemma \ref{lem: nonlinear norm Hk}, the first term of $N_D(\ep^D)$ defined in \eqref{difference: nonlinear term} can be estimated by
\begin{align*}
    &  \quad  \Big|\left( P_s \left( \na \cdot \left( \ep^D \na \Delta^{-1} \ep^{(0)} \right) \right), \ep_s^D \right)_{\dot H^2} \Big| \lesssim \| \ep^{(0)} \|_{H^2} \| \eps^D \|_{H^2} \| \ep^D \|_{H^3}.
\end{align*}

However, since we do not have enough $H^3$ estimate of each $\ep^{(i)}$, we need to put extra effort into the second term in \eqref{difference: nonlinear term}. Precisely,
\be
\begin{split}
  &\left( P_s \left( \na \cdot  \left(  \ep^{(1)} \na \Dein \ep^D \right) \right), \ep_s^D \right)_{\dot H^2}\\
     =&\left(  \na \cdot  \left(  \ep^{(1)} \na \Dein \ep^D \right), \ep_s^D \right)_{\dot H^2} - \left( P_u \left( \na \cdot \left( \ep^{(1)} \na \Dein \ep^D \right)\right), \ep_s^D \right)_{\dot H^2},\end{split}
     \label{H2 estimate: difference}
\ee
% where the second term, recalling $P_u: H^1 \to H^2$ is bounded from  the definition of $P_u$ with $k=1$ in  \eqref{Puinnerproduct}, can be bounded by
where the second term, recalling the boundedness $P_u: H^1 \to H^2$ from Proposition \ref{proposition: spectral properties of L} (2) with $k = 1$, can be bounded by
\begin{align*}
    & \quad \Big| \left( P_u \left(  \na \cdot  \left(  \ep^{(1)} \na \Dein \ep^D \right)  \right), \ep_s^D \right)_{\dot H^2} \Big|
    \le \Big\| P_u \left(  \na \cdot  \left(  \ep^{(1)} \na \Dein \ep^D \right)  \right) \Big\|_{H^2} \| \eps^D \|_{H^2}  \\
    & \le \Big\| \na \cdot  \left(  \ep^{(1)} \na \Dein \ep^D \right)   \Big\|_{H^1} \cdot  \| \ep_s^D \|_{H^2}
    \lesssim \| \ep^{(1)} \|_{H^2} \| \ep ^D\|_{H^2} \| \ep_s^D \|_{H^2},
\end{align*}
where the last inequality holds by the similar proof to Lemma \ref{lem: nonlinear norm Hk}. When it comes to the first term in \eqref{H2 estimate: difference}, by integration by parts and then following the same argument as Lemma \ref{lem: nonlinear norm Hk}, we obtain that
\begin{align*}
     \Big| \left( \na \cdot  \left(  \ep^{(1)} \na \Dein \ep^D \right) , \ep_s^D \right)_{\dot H^2} \Big| 
    % & = 
    % \Big| \left( \ep^{(1)}\ep^D, \ep_s^D \right)_{\dot H^2} \Big| 
    % + \Big| \left( \ep^{(1)} \na \Dein\ep^D, \na \ep_s^D \right)_{\dot H^2} \Big|  \\
    % & 
    \lesssim \| \ep^{(1)} \|_{H^2} \| \ep^D \|_{H^2} \| \ep_s^D \|_{H^3}.
\end{align*}
Consequently, we conclude that
\begin{align*}
    & \left( P_s \left( N_D(\ep^D) \right), \ep_s^D\right)_{\dot H^2}
    % \le C \left( \|\ep^{(0)} \| _{H^2}  +  \|\ep^{(1)} \| _{H^2}\right) \left( \| \eps^D \|_{H^2} \| \ep^D \|_{H^3} + \| \ep^D \|_{H^2} \| \ep_s^D \|_{H^3} \right) \\
     \le  C \left( \|\ep^{(0)} \| _{H^2}  +  \|\ep^{(1)} \| _{H^2}\right) \left( \| \eps^D \|_{H^2}^2 +  \| \epu^D \|_{\tilde B}^2 + \| \ep_s^D \|_{\dot H^3}^2 \right)\\
     \le & C \bar \delta \left( \calA^2 + \calE_1^2 + \calE_2^2 \right) e^{-\frac{3\delta_g}{2} \tau} +  C \bar \delta  \| \ep_s^D \|_{\dot H^3}^2,
\end{align*}
with $C>0$ a uniform constant. Recalling \eqref{eqdefcalAE}, this yields that
\begin{align*}
    & \quad \frac{1}{2} \frac{d}{d\tau} \| \ep_s^D \|_{\dot H^2}^2
     = \left( -\calL \ep_s^D + P_s \left( N_D(\ep^D) \right), \ep_s^D \right)_{\dot H^2} \\
    & \le -\frac{7}{8} \| \eps^D \|_{\dot H^3}^2 - \frac{5}{4} \| \eps^D \|_{\dot H^2}^2 + C_0^2 \| \eps^D\|_{H^1}^2 
    + C \bar \delta \left( \calA^2 + \calE_1^2 + \calE_2^2 \right) e^{-\frac{3\delta_g}{2} \tau} +  C \bar \delta  \| \ep_s^D \|_{\dot H^3}^2 \\
    & \le - \frac{5}{4} \| \eps^D \|_{\dot H^2}^2 + C_0^2 \calE_1^2 e^{-\delta_g \tau} +  C \bar \delta \left( \calA^2 + \calE_1^2 + \calE_2^2 \right) e^{-\frac{3\delta_g}{2} \tau},
\end{align*}
which, similar to the proof of Lemma \ref{lemimproveH5}, yields that
\be
\| \ep_s^D(\tau) \|_{\dot H^2}^2 \le e^{-\delta_g \tau} \| \ep_{s0}^D \|_{\dot H^2}^2 +  C_0^2 \calE_1^2 e^{-\delta_g \tau} +  C \bar \delta \left( \calA^2 + \calE_1^2 + \calE_2^2 \right) e^{-\frac{3\delta_g}{2} \tau}, \quad \forall \; \tau \ge 0.
\label{diffstaE2}
\ee

\vspace{0.5cm}

\noindent
\underline{\textit{Semigroup estimate of $\epu^D:= P_u \ep^D$.}} If we project the equation \eqref{eqepdif} onto $H_u^1(\RR^3)$, then $\epu^D$ solves the equation
\be
\partial_\tau \epu^D = \calL \epu^D  + P_u N_D(\ep^D),
\label{equnst}
\ee
and $\epu^D$ can be written as $\epu^D(\tau,\cdot) = \sum_{j=0}^4 a_j(\tau) \varphi(\cdot)$. Since $\{\psi_j\}_{j=0}^3$ satisfies \eqref{Puinnerproduct}, after taking inner product with $\psi_j$ in $H^1$ onto both sides of \eqref{equnst}, then $a_j$ solves the following ODE:
\begin{align*}
	\dot a_j = z_j a_j + \left( P_u N_D(\ep^D), \psi_j\right)_{H^1}, \quad \forall \; 0 \le j \le 3,
\end{align*}
with $z_0=\frac{1}{2}$ and $z_j=1$ with $1 \le j \le 3$, so $a_j$ can be written by
\begin{align}
& \quad a_j(\tau) 
 = e^{z_j \tau}a_j(0) + \int_0^\tau e^{z_j (\tau-s)}\left( P_u N_D(\ep^D), \psi_j\right)_{H^1}(s) ds \notag\\
& = e^{z_j \tau} \left( a_j(0) + \int_0^\infty e^{-z_j s} \left( P_u N_D(\ep^D), \psi_j\right)_{H^1}(s) ds \right) - \int_\tau^\infty e^{z_j (\tau-s)}\left( P_u N_D(\ep^D), \psi_j\right)_{H^1}(s) ds.
\label{difequnstable}
\end{align}
With the bootstrap assumption \eqref{bootstrapuns}, the unstable part must vanish, i.e.
\[
a_j(0) + \int_0^\infty e^{-z_j s} \left( P_u N_D(\ep^D), \psi_j\right)_{H^1}(s) ds=0,
\]
hence, by using \eqref{Lipnonlinearesti} and \eqref{eqdefcalAE}
\begin{align}
	|a_j(0)| &=  \Big| \int_0^\infty e^{-z_j s} \left( P_u N_D(\ep^D), \psi_j\right)_{H^1}(s) ds \Big| \notag\\
	& \lesssim \int_0^\infty e^{-z_j s}  \| \ep^D(s) \|_{H^2} \left(  \| \ep^{(0)}(s) \|_{H^2} + \| \ep^{(1)}(s) \|_{H^2}\right) ds 
	\notag\\
    % & \sim \int_0^\infty e^{-z_j s}  \| \ep^D(s) \|_{L^2} \left(  \| \ep^{(0)}(s) \|_{L^2} + \| \ep^{(1)}(s) \|_{L^2}\right) ds \notag\\
    & \lesssim \bar \delta (\calA + \calE_1 + \calE_2)\int_0^\infty e^{-z_j s} e^{-\delta_g s} ds \lesssim \bar\delta \left( \calA + \calE_1 + \calE_2 \right).
    \label{Lipaj0}
\end{align}
Similarly, from the equation \eqref{difequnstable}, the evolution of $a_j$ can be controlled by
\begin{align}
	|a_j(\tau)| 
	\le \int_\tau^\infty e^{ z_j (\tau-s)} \| \ep^D \|_{H^2} \left( \| \ep^{(0)} \|_{H^2} + \| \ep^{(1)} \|_{H^2} \right) ds
	\lesssim \bar \delta (\calA + \calE_1 + \calE_2) e^{-\frac{7}{10}\delta_g \tau}.
	\label{difesiunsta}
\end{align}

Consequently, combining the estimates \eqref{diffstaE1}, \eqref{diffstaE2}, \eqref{Lipaj0} and \eqref{difesiunsta} implies
\be
\begin{cases}
\calE_1 \le C \big\| \ep_{s0}^D \big\|_{H^1} + C\bar\delta (\calA + \calE_1 + \calE_2), & \quad \calE_2 \le C  \| \ep_{s0}^D \|_{\dot H^2} + C_0 \calE_1 +  C \sqrt{\bar \delta} \left( \calA + \calE_1+\calE_2 \right), \\
 \calA \le C\bar \delta (\calA + \calE_1 + \calE_2), & \quad \big\| \ep_{u0}^D \big\|_{\tilde B}  \le C \bar \delta (\calA + \calE_1 + \calE_2),
\end{cases}
\label{result estimate diff}
\ee
for some universal constants $C_0, C>0$. With $\bar \delta$ small enough, we have
\begin{align*}
    & \quad \calA + \calE_1 + \frac{1}{2 C_0} \calE_2  \\
    & \le C \bar \delta \left( \calA + \calE_1 + \calE_2 \right)
    +  C \big\| \ep_{s0}^D \big\|_{H^1} 
    + \frac{1}{2 C_0} \left( C  \| \ep_{s0}^D \|_{\dot H^2} + C_0 \calE_1 +  C \sqrt{\bar \delta} \left( \calA + \calE_1 +\calE_2 \right) \right) \\
    & \le \frac{1}{2} \left( \calA + \calE_1 + \frac{1}{2C_0}\calE_2 \right) + \left( C + \frac{C}{2 C_0} \right) \| \ep_{s0}^D \|_{H^2}, 
\end{align*}
which yields that
\[
\calA + \calE_1 + \calE_2 \lesssim \| \ep_{s0}^D \|_{H^2},
\]
with the coefficient uniformly in $\delta_j$ and $\ep^{(i)}$. Hence using \eqref{result estimate diff} again, we conclude
\begin{align*}
    \big\| \ep_{u0}^D \big\|_{\tilde B}  \lesssim \bar \delta (\calA + \calE_1 + \calE_2) \lesssim \bar{\delta} \| \ep_{s0}^D \|_{H^2},
    % \sim \bar{\delta} \| \ep_{s0}^D \|_{H^1},
\end{align*}
and hence we have concluded the proof.
\end{proof}

\subsection{Proof of Theorem \ref{thm: stability of ss solu}}
\label{subsefullsta}

First of all, based on Proposition \ref{propbootstrap} and Proposition \ref{propLipde}, we have the following finite-codimensional stability result: 
\begin{corollary}[Finite-codimensional stability]
\label{corcodsta}
        As for $\delta_0 \ll 1$ given in \eqref{bootstrapinitial}, there exists a Lipschitz map
        \[
    \Phi: B_{\delta_0}^{H^2(\RR^3)} (0)  \cap H_s^1(\RR^3) \to H_u^1(\RR^3),
    \]
    constructed in Proposition \ref{propLipde} and satisfying the bound \eqref{Lipsconti}, such that for any $\ep_{s0} \in B_{\delta_0}^{H^2(\RR^3)} (0)  \cap H_s^1(\RR^3) \to H_u^1(\RR^3)$, the solution to \eqref{equation, Keller-Segel} with initial data
    \[
    \rho_0 = Q + \ep_{s0} + \Phi(\ep_{s0}),
    \]
    blows up at $T=1$ with
    \be
    \rho(t,x) = \frac{1}{1-t} (Q+\ep)\left( t, \frac{x}{\sqrt{1-t}} \right), \quad \forall \;t \in (0,1),
    \label{finitecodista}
    \ee
    and
    \[
    \| \ep (t) \|_{H^2} \lesssim (1-t)^{\epsilon^*}, \quad \forall \; t \in (0,1),
    \]
    with $\epsilon_* >0$ independent of $\e_{s0}$.
\end{corollary}
\begin{proof}
    Recall the self-similar coordinate transformation given in \eqref{sscoordi}, if we choose $\mu(0) =1$, then $\mu$ can be explicitly solved by $\mu(\tau) = e^{-\frac12 \tau}$ on $\overline{\RR^+}$. Thus by chain rule,
    \[
    \frac{d\mu}{dt} = \frac{d\mu}{d\tau} \frac{d \tau}{dt} = - \frac{1}{2 \mu},
    \]
    which implies that the solution blows up at time $T=1$ in the original coordinate and $\mu$ can be explicitly solved by
    \[
    \mu(t) = \sqrt{1-t}, \quad \forall \; t \in [0,1).
    \]
    Then combining Proposition \ref{propbootstrap} and Proposition \ref{propLipde} yields the result.
\end{proof}

Next, for any initial data $\ep_0$ with $\| \ep_0 \|_{H^2} \ll 1$, we will find a scaling transformation and a spatial translation to shift it onto the stable manifold. This is the crucial step from Corollary \ref{corcodsta} towards Theorem \ref{thm: stability of ss solu}, by compensating the unstable codimensions via the symmetry group of \eqref{equation, Keller-Segel}.

\begin{lemma} 
\label{lemmatching}
Let $\Phi$ be the Lipschitz map in Corollary \ref{corcodsta} and $\delta_0$ from Proposition \ref{propbootstrap}. There exists $\delta \in (0, \delta_0)$, such that for any $\ep_0 \in H^2$ with $\| \ep_0 \|_{H^2} \le \delta$, there always exists $\vec{\alpha}_0 =(\lambda_0 -1, x_0)\in \RR^4$, with 
\be |\vec \alpha_0| \lesssim \| \ep_0\|_{H^2}, \label{eqsmalla0} \ee
such that the function
	\be
\ep_0^{\vec{\alpha_0}}(y) = \lambda_0^2(Q + \ep_0) (\lambda_0 y + x_0) - Q(y),
\label{modulateep}
\ee
satisfies 
\be \Phi \left((1-P_u)\ep_0^{\vec{\alpha_0}} \right) = P_u\ep_0^{\vec{\alpha_0}},
\label{fullstablematc}
\ee 
and
 \be 
 \big\| \ep_0^{\vec{\alpha_0}}  \big\|_{H^2} \le \delta_0.
 \label{smallnessofscaledep0}
 \ee
\end{lemma}
\begin{proof}
 We first show \eqref{eqsmalla0} and \eqref{fullstablematc} imply \eqref{smallnessofscaledep0}, so that the proof is reduced to finding $\vec \a_0$ solving \eqref{fullstablematc}. In fact, by Taylor expansion, with $\vec \Xi := (\Lambda Q, \nabla Q)$ and $\vec \a_0 := (\l_0 - 1, x_0)$, we have
\begin{align}
	\lambda_0^2Q  (\lambda_0 \cdot  + x_0) - Q
	&= \left( \frac{\pa \left( \lambda_0^2 Q(\lambda_0 \cdot + x_0) \right)}{\pa \vec \alpha_0} \Big|_{\vec{\alpha}_0=\vec 0}, \vec \alpha_0 \right)_{\RR^4} + G_{\vec \alpha_0} 
	= \vec \a_0 \cdot \vec \Xi +G_{\vec\alpha_0},
    \label{Taylorexpansionmatching}
\end{align}
where $\| G_{\vec{\alpha}_0} \|_{H^1} \lesssim |\vec \alpha_0|^2$, thus the $H^2$ norm of $\ep_0^{\vec{\alpha}_0}$ \eqref{modulateep} can be controlled by
\[
\big\| \ep_0^{\vec{\alpha}_0} \big\|_{H^2} \le \big\| \lambda_0^2 Q  (\lambda_0 \cdot  + x_0) - Q \big\|_{H^2} + \| \lambda_0^2 \ep_0  (\lambda_0 \cdot  + x_0) \|_{H^2} \lesssim |\vec\alpha_0|  + \| \ep_0 \|_{H^2}.
\]
Thereafter, \eqref{smallnessofscaledep0} follows from \eqref{eqsmalla0} and taking $\delta \ll \delta_0$.

Next, it suffices to verify the existence of $\vec{\alpha}_0$ satisfying \eqref{eqsmalla0} and \eqref{fullstablematc}. In fact, under the conditions given in this lemma, the equation \eqref{fullstablematc} is equivalent to
  \[
P_u \left( \lambda_0^2Q  (\lambda_0 \cdot  + x_0 ) - Q \right) = \Phi \left((1-P_u) \ep_0^{\vec{\alpha_0}} \right) - P_u \left( \lambda_0^2 \ep_0  (\lambda_0 \cdot  + x_0) \right),
\]
then using  \eqref{Taylorexpansionmatching} and $P_u \vec \Xi = \vec \Xi$, the equation \eqref{fullstablematc} can be further rewritten as
\be 
  \vec \a_0 \cdot \vec \Xi = \Phi \left((1-P_u) \ep_0^{\vec{\alpha}_0} \right) - P_u \left( \lambda_0^2 \ep_0  (\lambda_0 \cdot  + x_0) \right) - P_u G_{\vec{\alpha}_0} := \calF_{\vec{\alpha}_0,\ep_0}.
\ee
Recalling \eqref{Puinnerproduct}, after taking $H^1$ inner product with $\psi_j$ onto both sides of the equation, we obtain the following equation:
\[
\vec{\alpha}_0 = \calT \vec{\alpha}_0, 
\]
where $\calT \vec{\alpha}_0$, defined by
\[
 (\calT \vec{\alpha}_0)_0 := \frac{1}{\| \Lambda Q \|_{H^1}} \left( \calF_{\vec{\alpha}_0, \ep_0}, \psi_0 \right)_{H^1},
 \text{ and } 
 (\calT \vec{\alpha}_0)_j := \frac{1}{\| \partial_j Q \|_{H^1}} \left( \calF_{\vec{\alpha}_0, \ep_0}, \psi_j \right)_{H^1}  \; \text{ with } \; 1 \le j \le 3,
\]
is a continuous map and satisfies, via the Lipschitz dependence constructed in Proposition \ref{propLipde},
\begin{align*}
	|\calT \vec \alpha_0| \lesssim \| \calF_{\vec{\alpha}_0,\ep_0} \|_{H^1} \lesssim \bar \delta \| \ep_0^{\vec{\alpha}_0} \|_{H^1} + \| \ep_0 \|_{H^1} + \| G_{\vec{\alpha}_0} \|_{H^1} 
    \lesssim \| \ep_0 \|_{H^2} +
	\bar\delta |\vec \alpha_0| + |\vec \alpha_0|^2.
\end{align*}
Denote the universal constant as $C_*$. With $\bar\delta$ small enough from Proposition \ref{propbootstrap} and Proposition \ref{propLipde}, we can choose $\delta \ll \delta_0 \ll 1$ such that $C_*(\bar \delta + 2\delta C_*) < \frac 12$. Therefore
\[
\calT: B_{2C_* \| \e_0\|_{H^2}}^{\RR^4}(0) \to B_{C_* \| \e_0\|_{H^2}\left(1 + 2C_*(\bar \delta + 2C_* \| \e_0\|_{H^2}) \right)}^{\RR^4}(0) \subset B_{2C_* \| \e_0\|_{H^2}}^{\RR^4}(0),
\]
which implies that $\calT$ obtains a fixed point in $\vec \a_0 \in B_{2C_* \| \e_0\|_{H^2}}^{\RR^4}(0)$ via the Brouwer fixed point theorem and we thus have concluded the proof.
\end{proof}

Finally, we are ready to finish the proof of Theorem \ref{thm: stability of ss solu}.
\begin{proof}[Proof of Theorem \ref{thm: stability of ss solu}]
    Let $0< \delta \ll 1$ given in Lemma \ref{lemmatching}. For any $\ep_0 \in H^2$ with $\| \ep_0 \|_{H^2} \le \delta$, let $\rho^{\vec \alpha_0}$ be the solution to \eqref{equation, Keller-Segel} with initial data
    \[
    \rho_0^{\vec \alpha_0} := Q + \ep_0^{\vec{\alpha_0}}= \lambda_0^2(Q + \ep_0) (\lambda_0 \cdot + x_0),
    \]
    with $\vec \alpha_0$ determined in Lemma \ref{lemmatching}. Then Corollary \ref{corcodsta} and \eqref{fullstablematc} imply that $\rho^{\vec \a_0}$ blows up at time $T=1$ with
    \be
    \rho^{\vec \alpha_0}(t,x) = \frac{1}{1-t} (Q+\ep^{\vec \alpha_0})\left( t, \frac{x}{\sqrt{1-t}} \right), \quad \forall \;t \in (0,1),
    \label{finitecodista}
    \ee
    and
    \[
    \| \ep^{\vec \alpha_0} (t) \|_{H^2} \lesssim (1-t)^{\epsilon^*}, \quad \forall \; t \in (0,1).
    \]
    In addition, by invariance of \eqref{equation, Keller-Segel} and the local wellposedness of \eqref{equation, Keller-Segel}, the function
    \begin{align*}
      \rho(t,x) 
      &= \lambda_0^{-2} \rho_0^{\vec \alpha_0} \left( \frac{t}{\lambda_0^2},  \frac{x-x_0}{\lambda_0}\right) 
       =\frac{1}{\lambda_0^2 -t} \left( Q + \epsilon^{\vec \alpha_0} \right) \left( \frac{t}{\lambda_0^2},  \frac{x- x_0 }{ \sqrt{\lambda_0^2 -t} } \right),
    \end{align*}
    is the unique solution to \eqref{equation, Keller-Segel} with initial dat $\rho_0 = Q + \ep_0$, which satisfies \eqref{fullstabblowup} and \eqref{smallnessremainingterm} if we denote $\ep(t,\cdot ) := \ep^{\vec \alpha_0} \left( \frac{t}{\lambda_0^2}, \cdot \right)$, and hence we have finished the proof.
\end{proof}

\appendix
\section{Decay estimate of the unstable eigenfunction}\label{appA}
\begin{lemma}[Decay estimate of unstable eigenmodes]
\label{lemmadecayatinifty}
If $f \in H^2(\RR^3)$ is an eigenfunction of $\calL$ with respect to the eigenvalue $z$ with $\Re z < \frac{1}{4}$, namely, $\calL f = z f$, then we have the pointwise estimate of $f$ by
\be
|f(x)| \le C\left(\| f\|_{H^2},z \right) \frac{1}{\la x \ra^{\min\{ 2, 2(1-\Re z)- \}}}.
\label{pwofeigenfunction}
\ee
\end{lemma}

\begin{proof}
Since the unstable mode satisfies the equation
\[
\calL f= -\Delta f + \frac{1}{2} \Lambda f - 2 Q f - \nabla \Delta^{-1} Q \cdot \na f - \na \Delta^{-1} f \cdot \na Q = z f, 
\]
with $ \Re z < \frac{1}{4}$, 
then it means that
\[
\left( \calL_0 - z \right) f
= 2Qf + \na \Delta^{-1} Q \cdot \na f + \na Q \cdot \na \Delta^{-1} f, \quad \Rightarrow \quad f = (\calL_0 - z)^{-1} F,
\]
where $F$ is defined by 
\[
F= F_1 + F_2, \quad \text{ with } \quad  F_1 = 2Qf + \na Q \cdot \na \Delta^{-1} f, \quad  F_2 = \na f \cdot \na \Dein Q,
\]
and $(\calL_0 -z)^{-1}$, recalling the formula for $(\calL_0 -z)^{-1}$ from the proof of \cite[Lemma $2.2$]{ksnsblowup}, is given by
\be
f=(\calL_0 - z)^{-1} F = \int_0^\infty e^{-(1-z) \tau} \left( G_{1-e^{-\tau}} *F \right) \left( e^{-\frac{\tau}{2}} \cdot \right) d\tau,
\label{fintegralfomu}
\ee
with 
\be 
G_\lambda(x) = \lambda^{-\frac{3}{2}} G_1(\lambda^{-\frac{1}{2}} x )  \text{ and } G_1(x) =(4 \pi)^{-\frac{3}{2}} e^{-\frac{|x|^2}{4}}.
\label{Glambda}
\ee

\noindent \textit{Step 1. Pointwise estimate of $G_\lambda * F_1$.} First of all, let us consider the estimate of the nonlocal term $\na \Delta^{-1} f$, which can be written as
\[
\na \Delta^{-1} f =  \na \int_{\RR^3} \frac{Cf(y)}{|x-y|}dy  = C\int_{\RR^3} \frac{x-y}{|x-y|^3} f(y) dy,
\]
We split into $\{|y| \ge 2|x|\}$ and $\{ |y| \le 2|x|\}$ to compute 
\begin{align}
 &|\na \Delta^{-1} f(x)| \lesssim  \| f \|_{L^\infty} \int_{|y-x| \le 3 |x|} \frac{1}{|y-x|^2} dy  + \int_{|y| \ge |x|} \frac{1}{|y|^2} f(y) dy \nonumber \\
 \lesssim& |x|  \| f \|_{L^\infty} + \left\| |\cdot|^{-2} \right\|_{L^2(\RR^3 - B_{|x|})} \| f \|_{L^2} 
 \lesssim \left( |x| + |x|^{-\frac 12}\right)\| f \|_{H^2}.\nonumber 
\end{align}
Therefore with $|Q| \lesssim \la x \ra^{-2}$ and $|\nabla Q| \lesssim |x| \la x \ra^{-4}$ from \eqref{profile}, 
\be |F_1| \lesssim |Qf| + |\nabla Q \cdot \nabla \Delta^{-1} f| \lesssim \la x \ra^{-2} \| f \|_{H^2}. \label{Fpwdecay} \ee

Next, from the definition of $G_\lambda$ \eqref{Glambda} and \eqref{Fpwdecay}, we compute
\begin{align*}
&|G_\lambda * F_1(x)|  \lesssim \int_{\RR^3} \lambda^{-\frac{3}{2}} e^{-\frac{|\lambda^{-\frac{1}{2}}x-\lambda^{-\frac{1}{2}}y|^2}{4}} |F_1(y)| dy \\
 \lesssim& \int_{\RR^3} e^{-\frac{|\lambda^{-\frac{1}{2}} x -y|^2}{4}} |F_1(\lambda^\frac{1}{2} y)| dy 
 \lesssim \| f \|_{H^2} \int_{\RR^3} e^{-\frac{|\lambda^{-\frac{1}{2}} x -y|^2}{4}}	 \frac{1}{\la \lambda^\frac{1}{2}y \ra^2} dy.
\end{align*}
Similarly, we split the region of the integration into $\{|y- \lambda^{-\frac{1}{2}} x|  \le \frac{1}{2} |\lambda^{-\frac{1}{2}} x| \}$ and $\{|y- \lambda^{-\frac{1}{2}} x|  \ge \frac{1}{2} |\lambda^{-\frac{1}{2}} x| \}$ to estimate 
\begin{align}
  |G_\l* F_1(x)|& \lesssim \left[ \la x \ra^{-2} \int e^{-\frac{|\lambda^{-\frac{1}{2}} x -y|^2}{4}} dy  + \int_{|\tilde y| \ge |\l^{-\frac 12} x|} e^{-\frac{|\tilde y|^2}{4}} d\tilde y\right]  \| f \|_{H^2} \nonumber \\
  &\lesssim  \left(\la x \ra^{-2} +  e^{-C \lambda^{-1} |x|^2} \right)\| f \|_{H^2}. \label{Glambda*F}
\end{align}

\noindent \textit{Step 2. Estimate of $f_1:= (\calL_0 - z)^{-1} F_1$.} Combining  the integral expression of $f$ given in \eqref{fintegralfomu} and \eqref{Glambda*F}, it suffices to consider the integral
\begin{equation}
	\int_0^\infty e^{-(1-\Re z) \tau}  \left( \frac{1}{\la e^{-\frac{\tau}{2}}  x\ra^2} +  e^{-C (1-e^{-\tau})^{-1} e^{-\tau}|x|^2}  \right) d \tau.\label{eqintegralA2}
\end{equation} 
With the trivial bound $|\eqref{eqintegralA2}| \lesssim 1$, we now assume $|x| \ge 1$ and investigate the decay estimate for \eqref{eqintegralA2} below.

As for the first term of \eqref{eqintegralA2}, we find that for any $\epsilon' >0$ sufficiently small,
\begin{align}
    & \quad \int_0^\infty e^{-(1-\Re z) \tau}  \frac{1}{\la e^{-\frac{\tau}{2}}  x\ra^2}  d \tau
	\le  \int_0^\infty e^{-(1-\Re z) \tau} \frac{1}{\left(e^{-\tau} |x|^2 \right)^{ \min\{ 1, 1-\Re z -\frac{\epsilon'}{2}  \}} } d \tau  \notag \\
    & \lesssim \frac{1}{|x|^{\min\{ 2,2- 2\Re z - \epsilon' \}  }} \int_0^\infty e^{ - \frac{\ep'}{2}\tau} d \tau 
    \lesssim \frac{1}{|x|^{\min\{ 2,2- 2\Re z - \epsilon' \}  }}.
    \label{decayofeigenfunckey1}
\end{align}
For the second term in \eqref{eqintegralA2} we partition the integral range into $[0, 1]$ and $[1, \infty)$. 

When $\tau \in (0,1]$, we know that $\frac{e^{-\tau}}{1- e^{-\tau}} = \frac{1}{e^{\tau} -1} \sim \frac{1}{\tau}$, then
\begin{align}
	& \quad \int_0^1  e^{-(1-\Re z) \tau} e^{-\frac{C e^{-\tau}}{1-e^{-\tau}} |x|^2} d\tau 
	 \lesssim \int_0^1 e^{-\frac{C}{\tau} |x|^2} d\tau \notag\\
	 &=  \int_{|x|^2}^\infty  e^{-C s} s^{-2} |x|^2 ds \lesssim |x|^{2} e^{-\frac{C}{2}|x|} |x|^{-4} \int_{|x|^2}^\infty e^{-Cs} ds \lesssim  |x|^{-2} e^{-C'|x|}.
     \label{decayofeigenf2}
 \end{align}
 where we take a new variable $s=\frac{|x|^2}{\tau}$. 
 
 When $\tau \in (1, \infty)$, there exists $C'>0$ such that
 \begin{align}
	& \quad \int_1^\infty  e^{-(1-\Re z) \tau}  e^{-\frac{C e^{-\tau}}{1-e^{-\tau}} |x|^2} d\tau
	\lesssim \int_1^\infty e^{-(1-\Re z) \tau} e^{-C' |x|^2} d\tau \lesssim e^{-C' |x|^2}.
    \label{decayofeigenf3}
 \end{align}
 
 Consequently, the estimates above altogether deduce the pointwise estimate of $f_1$ using \eqref{fintegralfomu} by
 \be
 |f_1(x)|\lesssim \min \Big\{ \| f\|_{H^2}, \frac{1}{|x|^ {\min\{ 2, 2(1-\Re z)- \}} } \| f\|_{H^2} \Big\} \lesssim \frac{1}{\la x \ra^{ \min\{ 2, 2(1-\Re z)- \} }} \| f\|_{H^2}.
 \label{estioff1}
 \ee
 
 \noindent \textit{Step 3. Estimate of $f_2:= (\calL_0 - z)^{-1} F_2$.} We first remove the loss of derivative in $F_2$ through integration by parts
 \begin{align}
	G_{\lambda} * F_2(x)
	& = C\int \lambda^{-\frac{3}{2}} e^{-\frac{|\lambda^{-\frac{1}{2}}x-\lambda^{-\frac{1}{2}}y|^2}{4}} \na f(y) \cdot \na \Delta^{-1} Q(y) dy \notag\\
	& = \frac{C}{2} \lambda^{-\frac{1}{2}}\int \lambda^{-\frac{3}{2}} (\lambda^{-\frac{1}{2}} x - \lambda^{-\frac{1}{2}} y) e^{-\frac{|\lambda^{-\frac{1}{2}}x-\lambda^{-\frac{1}{2}}y|^2}{4}}  \cdot f(y) \na \Delta^{-1} Q(y) dy 
	\label{term1}\\
	& \quad - \frac{C}{2} \int \lambda^{-\frac{3}{2}} e^{-\frac{|\lambda^{-\frac{1}{2}}x-\lambda^{-\frac{1}{2}}y|^2}{4}}  f(y) Q(y) dy
	\label{term2}.
\end{align}
By the same argument as Step 1, 
\begin{align*}
|\eqref{term1}| &\lesssim \l^{-\frac 12} \int \l^{-\frac 32} e^{-\frac{ |\lambda^{-\frac{1}{2}}x-\lambda^{-\frac{1}{2}}y|^2}{8}}|f(y)| \la y \ra^{-1} dy \lesssim\lambda^{-\frac{1}{2}} \left(  \la x \ra^{-1} +  e^{-C\lambda^{-1} |x|^2} \right) \| f \|_{L^\infty}, \\
|\eqref{term2}| &\lesssim \left(  \la x \ra^{-2} +  e^{-C\lambda^{-1} |x|^2} \right) \| f \|_{L^\infty}, 
\end{align*}
for some constant $C > 0$. 
Here we emphasize that compared with \eqref{Glambda*F}, we lose one order of decay of $|x|$ for \eqref{term1} since $|\nabla \Delta^{-1} Q| \lesssim \la x \ra^{-1}$, and we have only used pointwise bound for $f$ so no longer require $\| f \|_{H^2}$.

Then to estimate $f_2$, it suffices to estimate the integral $I=I_1+I_2 + I_3$, where
\begin{align}
&I_1 =  \int_0^\infty e^{-(1-\Re z) \tau} (1-e^{-\tau})^{-\frac{1}{2}} \frac{1}{\la e^{-\frac{\tau}{2}}x \ra} d\tau, \label{term11} \\ 
& I_2 =  \int_0^\infty e^{-(1-\Re z) \tau} (1-e^{-\tau})^{-\frac{1}{2}}  e^{-\frac{1}{8} (1-e^{-\tau})^{-1} e^{-\tau} |x|^2}  d\tau, \label{term12} \\
& I_3=  \int_0^\infty e^{-(1-\Re z) \tau}  \left( \frac{1}{\la e^{-\frac{\tau}{2}}  x\ra^2} +  e^{-C (1-e^{-\tau})^{-1} e^{-\tau}|x|^2}  \right) d \tau.
\label{term13}
\end{align}
Since $(1-e^{-\tau})^{-\frac 12} \sim \tau^{-\frac 12}$ is integrable near $0$, we again have $|I| \lesssim 1$ and we will assume $|x| \ge 1$.
The last term $I_3$ has been handled in Step 2 and is bounded by $\frac{1}{\la x \ra^{ \min\{ 2, 2(1-\Re z)- \} }}$. As for $I_2$, we split the integral region into $(0,1)$ and $(1,\infty)$, and similar to the argument in \eqref{decayofeigenf2} and \eqref{decayofeigenf3}, $I_2$ can be bounded $|I_2| \lesssim e^{-C' |x|^2}$. As for $I_1$, similar to \eqref{decayofeigenfunckey1}, $I_1$ can be bounded by
\[
|I_1| \lesssim \int_0^\infty e^{-(1-\Re z) \tau} (1-e^{-\tau})^{-\frac{1}{2}} \frac{1}{|e^{-\frac{\tau}{2}} x |} d \tau \lesssim \frac{1}{|x|},
\]
where the last inequality follows from the fact that $1- \Re z > \frac{3}{4}$. Consequently, it yields that
\[ 
|f_2(x)| \lesssim \la x \ra^{-1} \| f \|_{L^\infty},
\]
and thus combining with the decay estimate of $f_1$ in Step $3$, $f$ can be estimated by
\be
|f(x)| \lesssim \la x \ra^{-1} \| f \|_{H^2}.
\label{decayoffstep1}
\ee

\mbox{}

\noindent \textit{Step 4.} Improve the decay of $f$. Recall the previous argument, the only obstacle to obtaining \eqref{pwofeigenfunction} is the constrained decay information of $|f(x) \na \Dein Q |$ when dealing with \eqref{term1}, where we originally can only gain $\frac{1}{\la x \ra}$ decay. However, with the decay rate of $f$ obtained in \eqref{decayoffstep1}, we are able to get one more order of the spatial decay of \eqref{term1} by
\[
|\eqref{term1}| \lesssim \lambda^{-\frac{1}{2}} \left( \la x \ra^{-2} + e^{-\frac{1}{8} \lambda^{-1} |x|^2} \right) \|\la x \ra f \|_{L^\infty}.
\]
Thus to estimate $f_2$, it suffices to estimate $I= I_1' +I_2 + I_3$, where $I_2$ and $I_3$ are given in \eqref{term12} and \eqref{term13} and $I_1'$ is defined by
\[
I_1' = \int_0^\infty e^{-(1-\Re z) \tau} (1-e^{-\tau})^{-\frac{1}{2}} \frac{1}{\la e^{-\frac{\tau}{2}}x \ra^2} d\tau,
\]
which, similar to \eqref{decayofeigenfunckey1}, can be bounded by  $|I_1'| \lesssim \frac{1}{|x|^{ \min \{2, 2-2 \Re z -\epsilon' \} }}$,
and then the decay of $f_2$ can be improved to $ |f_2(x)| \lesssim \frac{1}{|x|^{ \min \{2, 2-2 \Re z -\epsilon' \} }}$, which completes the proof in conjunction with \eqref{estioff1}.
\end{proof}

\section{Regularity of $\calH_{(l)}$ near the origin}\label{appB}
\begin{lemma}[Regularity of $\calH_{(l)}$ near the origion]
\label{lemflmrg}
	Fixed any $l\ge 0$ and $|m| \le l$, if $f(x) = f_{lm}(r) Y_{lm}(\Omega) \in H^\infty(\RR^3) \cap \calH_{(l)}$ with $l\ge 1$, then $f_{lm}$ can be decomposed into
	\be
	f_{lm}(r) = A_{lm}r^l + h_{lm},
	\label{flmdecayeq1}
	\ee
	for some constant $A_{lm}$ and function $h_{lm}$ satisfying
	\be
	|\partial_r^j h_{lm} (r)| \lesssim r^{l+2 -j} \text{ for any } r \in (0,1), \quad \text{ with }j =0,1,2.
	\label{flmdecayeq2}
	\ee
	\end{lemma}

\begin{proof}

   The assumption $f \in H^\infty(\RR^3) \cap \calH_{(l)}$  implies that   $F^{(j)} := \Delta^{j} f \in H^\infty(\RR^3) \cap \calH_{(l)}$ for any $j \ge 0$. For any $j \ge 1$, by \eqref{eqdefDeltal-1},
	\begin{align}
		F_{lm}^{(j-1)}(r) 
		& =  -\frac{1}{2l+1} r^{-(l+1)}\int_0^r F_{lm}^{(j)}(s) s^{l+2} ds - \frac{1}{2l+1} r^{l} \int_r^\infty F_{lm}^{(j)}(s) s^{-l+1} ds \notag \\
		& = -\frac{1}{2l+1} r^{-(l+1)}\int_0^r F_{lm}^{(j)}(s) s^{l+2} ds - \frac{1}{2l+1} r^{l} \int_1^\infty F_{lm}^{(j)}(s) s^{-l+1} ds \notag\\
		& \quad - \frac{1}{2l+1} r^{l} \int_r^1 F_{lm}^{(j)}(s) s^{-l+1} ds,
		\label{Flmeq}
	\end{align}
	where
	\begin{align}
		& \Big| r^{-(l+1)}\int_0^r F_{lm}^{(j)}(s) s^{l+2} ds \Big| \lesssim \| F_{lm}^{(j)} \|_{L^\infty}  r^{-(l+1)} \int_0^r s^{l+2} ds \lesssim r^2, 
		\label{flmterm1}\\
		& \Big|  r^{l} \int_1^\infty F_{lm}^{(j)}(s) s^{-l+1} ds \Big|
		\lesssim r^l \left(  \int_1^\infty |F_{lm}^{(j)}(s)|^2 s^2 ds \right)^\frac{1}{2} \left( \int_1^\infty s^{-2l} ds \right)^\frac{1}{2} \lesssim r^l, 
		\label{flmterm2}\\
		& \Big| r^l \int_r^1 F_{lm}^{(j)}(s) s^{-l+1} ds \Big| \lesssim
	r^l \| F_{lm}^{(j)}\|_{L^\infty} \int_r^1 s^{-l+1} ds \lesssim r^2 + r^l, \quad \text{ when } r \in (0,1).
	    \label{flmterm3}
	\end{align}
	Hence we have gained one more order of regularity of $F_{lm}^{(j-1)}$ near the origin. Precisely,
	\[
	|F_{lm}^{(j-1)}(r)| \lesssim r, \text{ for } r \in (0,1).
	\]
	
	With the estimate for $F_{lm}^{(j-1)}$ given above, we can use the same idea to obtain that regularity of $F_{lm}^{(j-k)}(r)$ near the origin for $k \in [0,j]$. Rigorously, as long as $k \le \min\{l,j \}$, for each iteration of the estimates, we can gain one more order of regularity near the origin, which follows from the similar argument to \eqref{flmterm1} and \eqref{flmterm3}. 
	
	Consequently, after iterating $k_0$ times with $k_0 \ge l +1$ with $j \ge l +1$, we can obtain the sharp regularity of $F_{lm}^{(j-k_0)}$ at the origin by
	\[
	|F_{lm}^{(j-k_0)}(r)| \lesssim r^l, \quad \text{ with } r \in (0,1).
	\] 
	In particular, we choose $j_0=l +1$ and iterate the estimate for $j_0$ times, then
	\[
	|f_{lm}(r)| = |F^{(j_0 -j_0)}(r)| \lesssim r^l, \quad \text{ with } r\in (0,1).
	\]
	More generally, we can obtain that for any $j \ge 0$, $F^{(j)} = F_{lm}^{(j)}(r) Y_{lm}(\Omega) := \Delta^jf \in H^\infty(\RR^3) \cap \calH_{(l)}$ satisfies
	\[
	|F_{lm}^{(j)}(r)| \lesssim r^l, \quad \text{ when } r \in (0,1).
	\]
	Hence using \eqref{eqdefDeltal-1} again,
	\begin{align}
		f_{lm}(r) & =-\frac{1}{2l+1} r^{-(l+1)}\int_0^r F_{lm}^{(1)}(s) s^{l+2} ds - \frac{1}{2l+1} r^{l} \int_r^\infty F_{lm}^{(1)}(s) s^{-l+1} ds \notag\\
		& = -\frac{1}{2l+1} r^{-(l+1)}\int_0^r F_{lm}^{(1)}(s) s^{l+2} ds
		- \frac{1}{2l+1} r^{l} \int_0^r F_{lm}^{(1)}(s) s^{-l+1} ds \notag \\
		& \quad - \frac{1}{2l+1} r^{l} \int_0^\infty F_{lm}^{(1)}(s) s^{-l+1} ds,
		\label{flmformula}
	\end{align}
	where
	\begin{align*}
	& \Big| r^{-(l+1)}\int_0^r F_{lm}^{(1)}(s) s^{l+2} ds \Big| 
	\lesssim r^{-(l+1)} \int_0^r s^{2l+2} ds \lesssim r^{l+2}, \\
	& \Big| r^l \int_0^r F_{lm}^{(1)}(s) s^{-l+1} ds  \Big| \lesssim r^{l} \int_0^r s ds \lesssim r^{l+2}, \\
	& \Big| \int_0^\infty F_{lm}^{(1)}(s) s^{-l+1} ds \Big| \lesssim \int_0^1 s^2 ds + \left( \int_1^\infty F_{lm}^2(s) s^2 ds  \right)^\frac{1}{2} \left(  \int_1^\infty s^{-2l} ds \right)^\frac{1}{2} \lesssim 1,
	\end{align*}
    which yields a decomposition of $f_{lm}$ as
	\begin{align*}
		f_{lm}(r) = A_{lm}r^l + h_{lm}(r), \quad \text{ with } |h_{lm}(r)| \lesssim r^{l+2} \text{ near the origin},
	\end{align*}
	with
	\[
	A_{lm} =-\frac{1}{2l+1} \int_0^\infty F_{lm}^{(1)}(s) s^{-l+1} ds.
	\]
	
	As for $\partial_r^j f_{lm}(r)$, applying \eqref{flmformula} again and using the similar argument as given above, it can be easily verified that 
	\begin{align*}
		\partial_r^j f_{lm}(r) = l(l-1)...(l-j+1)A_{lm}r^{l-j} + \partial_r^j h_{lm} (r), \text{ with } |\partial_r^j h_{lm} (r)| \lesssim r^{l+2-j} \text{ near the origin}.
	\end{align*}
\end{proof}

\begin{lemma}
\label{lemflmL2integrable}
	Fixed any $l\ge 1$ and $|m| \le l$, if the function $f(x) = f_{lm}(r) Y_{lm} (\Omega) \in H^\infty(\RR^3) \cap \calH_{(l)}$, then 
	\[
	\partial_r^j f_{lm} \in  L^2(\RR^+;r^2 dr) , \quad \text{ with } j=0,1,2,
	\]
	and
		\[
	\partial_r^j f_{lm} \in  L^2(\RR^+) , \quad \text{ with } j=0,1.
	\]
\end{lemma}
\begin{proof}
	First of all, $f_{lm} \in L^2(\RR^+;r^2 dr)$ is a direct result of  $L^2$ decomposition \eqref{directdecomofL2}. In addition, using the fact that 
	\[
	\| f_{lm}(r)\|_{L^\infty(\RR^+)} \lesssim \| f \|_{L^\infty} \lesssim \| f \|_{H^2},
	\]
	we find that $f_{lm} \in L^2(\RR^+; \la r \ra^2 dr)$, and thus $f_{lm} \in L^2(\RR^+)$.
	
	 When it comes to $\partial_r f_{lm}$, 
	recalling \eqref{nablaflm}, 
\[
 \| \na f \|_{L^2(\RR^3)}^2 = \int_0^\infty |\partial_r f_{lm}|^2 r^2 dr + l(l+1) \int_0^\infty f_{lm}^2(s) s^2 ds < \infty,
\]
which yields that $\partial_r f_{lm} \in L^2(\RR^+; r^2 dr)$. Together with the smoothness of $\partial_r f_{lm}$ near the origin (see Lemma \ref{lemflmrg}), we know that $\partial_r f_{lm}  \in L^2(\RR^+; \la r \ra^2 dr)$, and thus $\partial_r f_{lm} \in L^2(\RR^+)$.

As for $\partial_r^2 f_{lm}$, since $\Delta$ acts invariantly on $\calH_{(l)}$, we find that
\[
\partial_r^2 f_{lm} + \frac{2}{r} \partial_r f_{lm} - \frac{l(l+1)}{r^2} f_{lm} = \Delta_1 f_{lm} = (\Delta f)_{lm} \in L^2(\RR^+; r^2 dr).
\]
Note that
\begin{align*}
	\Big\| \frac{\partial_r f_{lm}}{r} \Big\|_{L^2(\RR^+;r^2 dr)}^2 
	= \int_0^\infty |\partial_r f_{lm}|^2 dr = \| \partial_r f_{lm} \|_{L^2(\RR^+)}^2 <\infty,
\end{align*}
and 
\begin{align*}
	\Big\|\frac{f_{lm}}{r^2} \Big\|_{L^2(\RR^+;r^2 dr)}^2 
	& = \int_0^\infty \frac{|f_{lm}|^2}{r^2} dr = \int_0^1 \frac{|f_{lm}|^2}{r^2} dr + \int_1^\infty |f_{lm}|^2 dr \lesssim 1, \quad \text{ with } l \ge 1,
\end{align*}
where the last inequality follows from the smoothness of $f_{lm}$ near the origin (see Lemma \ref{lemflmrg}), we deduce that $\partial_r^2 f_{lm} \in L^2(\RR^+; r^2 dr)$.
\end{proof}

\normalem
\bibliographystyle{siam}
\bibliography{Bib-2}

\end{document}